\documentclass[11pt,oneside,a4paper]{amsart}

\usepackage{lmodern}
\usepackage{color}
\usepackage{amssymb}
\usepackage{amsmath}
\usepackage{setspace}
\usepackage{marginnote}
\usepackage{graphicx}
\usepackage{float}
\usepackage[pagewise]{lineno}

\usepackage[english]{babel}

\usepackage[dvipsnames]{xcolor}

\usepackage{soul}

\newtheorem{thm}{Theorem}[section]
\newtheorem{prop}[thm]{Proposition}
\newtheorem{lem}[thm]{Lemma}

\theoremstyle{definition}
\newtheorem{defn}[thm]{Definition}

\theoremstyle{remark}
\newtheorem{exam}{Example}
\newtheorem{claim}{Claim}
\newtheorem*{claim*}{Claim}
\newtheorem*{rmk}{Remark}

\newcommand{\barre}[1]{\overline{#1}}

\newcommand{\p}{\partial}
\newcommand{\R}{\mathbb{R}}

\newcommand{\T}{\mathbb{T}}
\newcommand{\Z}{\mathbb{Z}}
\newcommand{\boC}{\mathcal{C}}

\newcommand{\boN}{\mathcal{N}}
\newcommand{\boF}{\mathcal{F}}
\newcommand{\boV}{\mathcal{V}}
\newcommand{\boO}{\mathcal{O}}
\newcommand{\boG}{\mathcal{G}}
\newcommand{\boY}{\mathcal{Y}}
\newcommand{\boX}{\mathcal{X}}
\newcommand{\boZ}{\mathcal{Z}}
\newcommand{\boT}{\mathcal{T}}
\newcommand{\boS}{\mathcal{S}}
\newcommand{\gbf}{\mathbf{g}}
\newcommand{\kbf}{\mathbf{k}}
\newcommand{\tbf}{\mathbf{t}}
\newcommand{\vbf}{\mathbf{v}}
\renewcommand{\S}{\mathbb{S}}
\DeclareMathOperator{\Span}{span}
\DeclareMathOperator{\tr}{tr}
\DeclareMathOperator{\id}{id}
\DeclareMathOperator{\out}{out}
\DeclareMathOperator{\tann}{tan}
\DeclareMathOperator{\diam}{Diam}
\DeclareMathOperator{\length}{Length}

\renewcommand{\tan}{\tann}

\usepackage[margin=2.5cm]{geometry}

\title[Rigidity of min-max minimal disks in $3$-balls with non-negative Ricci 
curvature]{Rigidity 
of min-max minimal disks in $3$-balls with non-negative Ricci curvature}

\author{Laurent Mazet}
\address{Institut Denis Poisson, CNRS UMR 7013, Universit\'e de Tours, 
Universit\'e d'Orl\'eans, Parc de Grandmont, 37200 Tours, France}
\email{laurent.mazet@univ-tours.fr}
\author{Abra\~ao Mendes}
\address{Instituto de Matem\'atica, Universidade Federal de Alagoas, Macei\'o, AL, Brazil}
\email{abraao.mendes@im.ufal.br}
\thanks{L.M. was partially supported by the ANR-19-CE40-0014 grant. A.M. was partially supported by the National Council for Scientific and Technological Development – CNPq, Brazil (Grant 305710/2020-6) and the Funda\c{c}\~ao de Amparo \`a Pesquisa do Estado de Alagoas – FAPEAL, Brazil (Process E:60030.0000002254/2022). The authors were partially supported by the Coordena\c{c}\~ao de Aperfei\c{c}oamento de Pessoal de N\'ivel Superior – Brasil (CAPES-COFECUB 88887.143161/2017-0).}

\numberwithin{equation}{section}

\DeclareMathOperator{\Ric}{Ric}
\DeclareMathOperator{\Div}{div}

\newcommand{\boH}{\mathcal{H}}
\newcommand{\Ome}{\Omega}
\newcommand{\eps}{\varepsilon}
\newcommand{\ii}{\mathrm{II}}
\newcommand{\Diff}{\mathrm{Diff}}


\usepackage{setspace}

\begin{document}

\raggedbottom

\linespread{1.25}

\begin{abstract}
In this paper we prove a rigidity statement for free boundary minimal surfaces 
produced via min-max methods. More precisely, we prove that for any Riemannian 
metric $g$ on the 3-ball $B$ with non-negative Ricci curvature and $\ii_{\p 
B}\ge g_{|\p B}$, there exists a free boundary 
minimal disk $\Delta$ of least area among all free boundary minimal disks in 
$(B,g)$. Moreover, the area of any such $\Delta$ equals to the width of 
$(B,g)$, $\Delta$ has index one, and the length of $\p\Delta$ is bounded from 
above by $2\pi$. Furthermore, the length of $\p\Delta$ equals to $2\pi$ if and 
only if $(B,g)$ is isometric to the Euclidean unit ball. This is related to a 
rigidity result obtained by F.C.~Marques and A.~Neves in the closed case. The 
proof uses a rigidity statement concerning half-balls with non-negative Ricci curvature 
which is true in any dimension. 
\end{abstract}

\maketitle

\section{Introduction}

In Riemannian geometry, a classical question consists in controlling the size 
of 
a Riemannian manifold in terms of its curvature tensor. For example, 
Bonnet-Myers theorem gives an upper bound on the diameter of a Riemannian 
manifold under a lower bound on its Ricci curvature. The proof of this result 
is based on the interplay between the Ricci curvature and the minimization 
property of some geodesics.

Other notions for the size of a manifold can be considered. For example, for a 
nontrivial homology class, one could think of the minimal volume of 
elements 
in this class. For closed curves this leads to the notion of systole and gives 
estimate for the injectivity radius, see for example Klingenberg's estimate or 
Toponogov theorem.

For submanifolds of dimension at least $2$, stable minimal submanifolds appear
as realizations of the minimum of the volume in the homology class. Hence 
estimating the volume of such stable, or even area-minimizing, minimal 
submanifolds is a natural question. For example, H.~Bray, S.~Brendle 
and 
A.~Neves~\cite{BrBrNe} have 
proved that an area-minimizing $2$-sphere in a $3$-manifold whose scalar 
curvature is at least $2$ has area at most $4\pi$. Moreover, in case of 
equality, the universal cover of the ambient manifold is isometric to the 
standard cylinder 
$\S^2\times\R$.

Another concept to measure the size of a Riemannian manifold is a min-max 
quantity $W$ called the 
width introduced by F.~Almgren to produce minimal hypersurfaces when the 
topology 
of the ambient manifold does not allow the minimization approach (see also 
M.~Gromov~\cite{Gro2}). In this situation, the minimal hypersurfaces that 
appear 
are in 
general neither area-minimizing nor stable. For example, in the case of a 
Riemannian $3$-sphere, L.~Simon and F.~Smith~\cite{Smi} defined a notion of 
width that allowed them to produce a minimal $2$-sphere. In a seminal work, 
F.C.~Marques and A.~Neves~\cite{MaNe3} initiated the study 
of the Morse index of minimal hypersurfaces produced by min-max methods. Among 
other things, they proved the following rigidity result:

\begin{thm}[{\cite[Theorem~4.9]{MaNe3}}]\label{th:MaNe}
Let $g$ be a Riemannian metric on the 3-sphere $\S^3$, with scalar curvature 
$R\ge6$, such that there are no stable embedded minimal spheres in $(\S^3,g)$ 
(in 
particular, if $g$ has positive Ricci curvature). Then there is an 
embedded 
minimal sphere $\Sigma$ in $(\S^3,g)$ such that 
\[
|\Sigma|=\inf\{|S|;S\text{ is an embedded minimal sphere in }(\S^3,g)\}.
\]
Moreover, any such $\Sigma$ satisfies the following conditions:
\begin{itemize}
\item $|\Sigma|=W(\S^3,g)$ (the Simon-Smith width);
\item $\Sigma$ has index one;
\item $|\Sigma|\le4\pi$.
\end{itemize}
Besides, the equality $|\Sigma|=4\pi$ holds if and only if $(\S^3,g)$ is 
isometric to the Euclidean unit sphere.
\end{thm}

When the sectional curvature of $(\S^3,g)$ is at most $1$, the first 
author~\cite{Maz29} proved that 
we have the 
reverse estimate: the width is at least $4\pi$ and the rigidity statement is 
also true in the equality case.

In the case of compact Riemannian manifold with boundary, the min-max approach 
has also been used to produce free boundary minimal hypersurfaces: minimal 
hypersurfaces that meet the boundary orthogonally (see \cite{CaFrSc,Ket2}). In 
this context, one could 
also expect some estimates of the width under some curvature assumptions. In 
this paper, we consider the case of a Riemannian $3$-ball $(B,g)$. The 
Simon-Smith approach can be used to produce either a free boundary 
minimal disk or a minimal sphere under a convexity assumption on the boundary 
(see M.~Gr\"{u}ter and J.~Jost~\cite{GrJo1,Jos2} and precise definitions 
below). 
Besides, if the Ricci curvature is 
non-negative, the 
minimal sphere can be excluded. Under these assumptions, the minimal disk 
should have index one and computation does not lead to an estimate of the area 
of the disk but of its perimeter. Actually, the second 
author~\cite[Theorem~1.3]{Men} 
proved
\begin{thm}\label{th:Men}
Let $(M^3,g)$ be a compact orientable Riemannian $3$-manifold with boundary 
with non-negative Ricci curvature and such that $\ii_{\partial M}\ge 
g_{|\partial 
M}$. If $\Sigma$ is an orientable free boundary minimal surface in $M$ of index 
one, then its perimeter satisfies
\[
L(\partial\Sigma)\le 2\pi(\gbf+\kbf),
\]
where $\gbf$ is the genus of $\Sigma$ and $\kbf$ the number of connected 
components 
of $\partial\Sigma$. Moreover, in case of equality,
\begin{enumerate}
\item $\Sigma$ with its induced metric is isometric to the Euclidean unit disk;
\item $\partial\Sigma$ is a geodesic of $\partial M$;
\item $\Sigma$ is totally geodesic;
\item all sectional curvatures of $M$ vanish on $\Sigma$.
\end{enumerate}
\end{thm}

Moreover, he was able to prove that, in the equality case and under some extra 
hypothesis, $(M,g)$ is isometric to the Euclidean unit $3$-ball. Let us remark 
that under the curvature assumptions on the Ricci tensor and $\ii_{\partial 
M}$, $(M^3,g)$ in the above theorem is diffeomorphic to the $3$-ball (see the 
work of A.~Fraser and M.M.-C.~Li~\cite[Theorem~2.11]{FrLi}).
One of the main results of the paper should be compared with 
Theorem~\ref{th:MaNe}; it makes the link between the min-max 
construction of free boundary minimal disks and Theorem~\ref{th:Men}. Moreover 
we 
are able to obtain the rigidity statement without any extra hypothesis.

\begin{thm}\label{th:main}
Let $g$ be a Riemannian metric on the 3-ball $B$ with non-negative Ricci curvature and such that $\ii_{\partial B}\ge g_{|\partial B}$, where $\ii_{\partial B}$ is the second fundamental form of $\partial B$ in $(B,g)$ with respect to the inward unit normal. Then there is a free boundary minimal disk $\Delta$ in $(B,g)$ such that
\[
|\Delta|=\inf\{|D|;D\text{ is a free boundary minimal disk in }(B,g)\}.
\]
Moreover, any such $\Delta$ satisfies the following conditions:
\begin{itemize}
\item $|\Delta|=W(B,g)$;
\item $\Delta$ has index one;
\item $L(\partial\Delta)\le 2\pi$.
\end{itemize}
Besides, the equality $L(\partial\Delta)=2\pi$ holds if and only if $(B,g)$ is isometric to the Euclidean unit ball.
\end{thm}


Actually, it would be interesting to have an estimate of the area of the free 
boundary minimal disk and then of the width $W(B,g)$. When the non-negativity 
of the Ricci curvature is replaced by the non-negativity of the sectional 
curvature, we 
are able to prove that $\Delta$ has area at most~$\pi$. Actually, under this 
assumption, we can prove that $|\Sigma|\le \frac12 L(\partial\Sigma)$ for any 
free boundary minimal surface $\Sigma$.

The proof of the rigidity statement is based on a characterization of the 
Euclidean unit half-ball $B_+^n=\{(x_1,\dots,x_n)\in\R^n\mid x_1^2+\cdots+x_n^2\le1\text{ and }x_n\ge 0\}$ which is true 
in any dimension (see Theorem~\ref{th:rigidhalfball} for a precise statement).
\begin{thm}\label{th:rigid}
Let $g$ be a Riemannian metric on $B_+^n$ with non-negative Ricci curvature and 
such that the mean curvature of $\S^n\cap B_+^n$ is at least $n-1$. We also 
assume that 
$B_+^n\cap\{x_n=0\}$ is totally geodesic, isometric to the Euclidean unit 
$(n-1)$-ball and meets $\S^n\cap B_+^n$ orthogonally. Then $(B_+^n,g)$ is 
isometric to the Euclidean $B_+^n$.
\end{thm}

The proof of the above result is based on ideas that appear in the work of 
R.~Reilly~\cite{Rei} and also used by C.~Xia~\cite{Xia}.

\subsection*{Organization of the paper} In Subsection~\ref{subsec:definitions} 
we recall some important definitions for the context of free boundary minimal 
hypersurfaces. In Subsection~\ref{subsec:nicefoliation} we introduce and prove 
the existence of the \textit{nice foliation} associated with a free boundary 
minimal hypersurface of index at least one. In 
Subsection~\ref{subsec:areaofFBMH} we state and prove an independent result 
that furnishes a sharp area estimate for free boundary minimal hypersurfaces 
in 
terms of their perimeter when the ambient Riemannian manifold has 
non-negative sectional curvature and strictly convex boundary. As an 
application, if $g$ in Theorem~\ref{th:main} has non-negative sectional 
curvature, then $|\Delta|\le\pi$ and the equality holds if and only if $(B,g)$ 
is isometric to the Euclidean unit ball. In Section~\ref{sec:minmax} we 
recall 
some min-max constructions of free boundary minimal surfaces in Riemannian 
$3$-manifolds with strictly convex boundary. There, we study two different 
cases: the first one applies to compact $3$-manifolds with smooth boundary, 
this is called the \textit{unconstrained case} 
(Subsection~\ref{subsec:unconstrained}); the second deals with compact 
$3$-manifolds with piecewise smooth boundary and, in this case, we want to 
prevent the min-max minimal surface to attach to a certain part of the 
boundary, this is called the \textit{constrained case} 
(Subsection~\ref{subsec:constrained}). In Section~\ref{sec:topologicalcontrol} 
we explain how to control the topology of the minimal surface obtained through 
the min-max constructions in the preceding section. Notice that a similar 
topological control has been also obtained very recently by G.~Franz and 
M.~Schulz \cite{FrSc}. In 
Section~\ref{sec:rigid} we state and prove a rigidity result for the 
Euclidean half-ball (Theorem~\ref{th:rigidhalfball}), which plays an important 
role in the proof of Theorem~\ref{th:main}. Finally, in 
Section~\ref{sec:least} we use all the machinery presented in the previous 
sections to prove Theorem~\ref{th:main}. We end the paper with 
Appendix~\ref{sec:foliat} where 
we construct some mean-convex foliation under some geometric hypotheses. We 
don't use this construction in the paper but it allows to apply the min-max 
theory in some more general situations.

\subsection*{Notations} Among other notations, we will use $\boH^2$ to denote 
the $2$-dimensional Hausdorff measure. If $S$ is a $k$-submanifold of a 
Riemannian manifold, we will denote by $\vbf(S)$ the associated 
$k$-dimensional varifold.

\subsection*{Thanks} The authors would like to thank Martin Li for his answers 
about the min-max theory in the free boundary setting. They also thanks 
Robert Haslhofer for pointing us the reference~\cite{HaKe}.

\section{Preliminaries}\label{sec:prel}

\subsection{Definitions}\label{subsec:definitions}

We say that $M$ is an \textit{$n$-manifold with piecewise smooth boundary} if 
it has local $C^\infty$-charts given by open subsets of $\R_+^2\times\R^{n-2}$. 
The set of points in $\partial M$ corresponding to $\{(0,0)\}\times\R^{n-2}$ by 
the charts is called the \textit{corner} of $M$ and denoted by $\boC(M)$. 

Let $M$ be such an $n$-manifold with piecewise smooth boundary and $\Sigma$ be 
an $(n-1)$-manifold with boundary. We say that a smooth embedding 
$\phi:\Sigma\to M$ is \textit{proper} if
\[
\phi(\partial\Sigma)=\phi(\Sigma)\cap \partial M\subset\partial M\setminus\boC(M).
\]
If $\Sigma\subset M$ and $\phi$ is just the inclusion map, we say that $\Sigma$ is a \textit{properly embedded hypersurface} in~$M$.

Let $M$ be endowed with a Riemannian metric. If $\Sigma$ is a properly embedded 
hypersurface and $\{F_t\}$ is a smooth family of proper embeddings of $\Sigma$ 
in $M$ with $F_0=\id$, one can compute the $(n-1)$-volume of $F_t(\Sigma)$ 
(denoted by $|F_t(\Sigma)|$) and 
its derivative with respect to $t$ at time $t=0$. We have 
\[
\frac{d}{dt}|F_t(\Sigma)|_{|t=0}=-\int_\Sigma(X,\vec{H})+\int_{\partial\Sigma}(X,\nu),
\]
where $X=\frac{\p}{\p t}{F_t}_{|t=0}$ is the variation vector field, $\vec{H}$ is the mean curvature vector of $\Sigma$ in $M$ and~$\nu$ is the unit conormal to $\partial\Sigma$ in $\Sigma$. So $\Sigma$ is critical for the $(n-1)$-volume if its mean curvature vector vanishes and $\Sigma$ meets $\partial M$ orthogonally. We call such a hypersurface a \textit{free boundary minimal hypersurface}. More generally, a properly embedded hypersurface that meets~$\partial M$ orthogonally is called a \textit{free boundary hypersurface}.

Except at the corner, we denote by $\ii_{\partial M}$ the second fundamental 
form of the boundary of~$M$ with respect to the inward unit normal. If $\Sigma$ 
is a free boundary minimal hypersurface, and assuming that $\Sigma$ is 
two-sided and $X$ is normal to $\Sigma$, one can compute the second derivative 
of the $(n-1)$-volume functional: 
\begin{equation}\label{eq:stab}
\begin{split}
\frac{d^2}{dt^2}|F_t(\Sigma)|_{|t=0}&=\int_{\Sigma}-u(\Delta u+(\Ric(N,N)+\|A\|^2)u)+\int_{\partial\Sigma}u(\partial_\nu u-\ii_{\partial M}(N,N)u)\\
&=\int_{\Sigma}\|\nabla u\|^2-(\Ric(N,N)+\|A\|^2)u^2-\int_{\partial\Sigma}\ii_{\partial M}(N,N)u^2, 
\end{split}
\end{equation}
where $N$ is the unit normal to $\Sigma$, $u=(X,N)$, $A$ is the Weingarten map on $\Sigma$, and $\Ric$ is the Ricci tensor of $M$. Our sign convention is that in which $\vec{H}=HN$, where $H=\tr A$ is the mean curvature of $\Sigma$ with respect to $N$. Denoting by $Q(u,u)$ the last line in \eqref{eq:stab}, we obtain a quadratic form associated with the Jacobi operator $L$ of $\Sigma$,
\[
Lu=-\Delta u-(\Ric(N,N)+\|A\|^2)u.
\]
The index of the quadratic form $Q$ is called the \textit{index} of $\Sigma$ and it is given by the number of negative eigenvalues of $L$ with a Robin type boundary condition:
\[
\begin{cases}
Lu=\lambda u&\text{on }\Sigma,\\
\partial_\nu u-\ii_{\partial M}(N,N)u=0&\text{on }\partial\Sigma.
\end{cases}
\]

\subsection{The nice foliation}\label{subsec:nicefoliation}

Let $\Sigma$ be a free boundary minimal hypersurface in $M$. If $\Sigma$ has index at least one, we are going to construct a foliation of a tubular neighborhood of $\Sigma$ by free boundary hypersurfaces whose mean curvature vectors are nowhere vanishing and point away from $\Sigma$.

Near $\Sigma$, let us parametrized $M$ by $\Sigma\times (-1,1)$ with 
coordinates $(p,t)\in \Sigma\times (-1,1)$ such that the variation vector field 
$\p_t$ is normal to $\Sigma$ at $t=0$. If $u$ is a function defined on $\Sigma$ 
with small $L^\infty$-norm, we can consider its graph: the image of 
$X_u:p\in\Sigma\mapsto(p,u(p))\in\Sigma\times(-1,1)\subset M$. For such a 
function $u$, we denote by $H(u)(p)$ the mean curvature of $X_u(\Sigma)$ at 
$X_u(p)$ with respect to the unit normal that points in the same direction as 
$\partial_t$. We have the following result:

\begin{lem}\label{lem:nicefol}
With the above notations, there is a smooth family $\{u_t\}_{t\in(-\eps,\eps)}$ of functions on $\Sigma$ such that
\begin{itemize}
\item $u_0=0$ and $\partial_t u_t>0$;
\item $H(u_t)$ is nowhere vanishing and has the same sign as $t$ for $t\neq 0$;
\item $X_{u_t}(\Sigma)$ is a free boundary hypersurface.
\end{itemize}
\end{lem}

With $\{u_t\}_{t\in(-\eps,\eps)}$ as in the above Lemma, the family $\{X_{u_t}(\Sigma)\}_{t\in(-\eps,\eps)}$ gives a foliation of a neighborhood of $\Sigma$ which we call the \emph{nice foliation} associated with $\Sigma$.

\begin{proof}
Since $\Sigma$ has index at least one, the first eigenvalue of the Jacobi operator is negative: there are $\lambda_1<0$ and $\phi_1$ a positive solution to
\[
\begin{cases}
\Delta\phi_1+(\Ric(N,N)+\|A\|^2)\phi_1+\lambda_1\phi_1=0&\text{on }\Sigma,\\
\partial_{\nu}\phi_1-\ii_{\partial M}(N,N)\phi_1=0&\text{on }\partial\Sigma.
\end{cases}
\]

Let $V^{k,\alpha}$ denote the $L^2$-orthogonal complement of $\phi_1$ in $C^{k,\alpha}(\Sigma)$ and $\Pi$ be the orthogonal projection onto $V^{k,\alpha}$. Let $\eta$ denote the inward unit normal to $\partial M$. Then, for $u$ a function on $\Sigma$ and $p\in\partial\Sigma$, we define $\boN(u)(p)=\det(dX_u(e_1),\ldots,dX_u(e_{n-1}),\eta(X_u(p)))$, where $(e_i)$ is a direct orthonormal basis of $T_p\Sigma$ (this does not depend on the choice of the basis). We then define the map
\[
F:\begin{array}{ccc}
\R\times V^{2,\alpha}&\longrightarrow&V^{0,\alpha}\times C^{1,\alpha}(\partial\Sigma),\\
(t,v)&\longmapsto&\Big(\Pi(H(t\phi_1+v))+\lambda_1 v,\boN(t\phi_1+v)\Big).
\end{array}
\] 
Notice that $F(0,0)=0$. We want to apply the implicit function theorem to $F$. 
Standard calculations show that (see \cite[Proposition~17]{Amb})
\[
D_vF(0,0)(h)=\big(-\Pi(Lh)+\lambda_1h,\partial_\nu h-\ii_{\partial M}(N,N)h\big).
\]
Let us see that $D_vF(0,0)$ is invertible. In fact, if $D_vF(0,0)(h)=0$, we have
\[
\int_\Sigma\phi_1Lh=\int_\Sigma hL\phi_1+\int_{\partial\Sigma}(-\phi_1\partial_\nu h+h\partial_\nu\phi_1)=\lambda_1\int_\Sigma h\phi_1-\int_{\partial\Sigma}\phi_1(\partial_\nu h-\ii_{\partial M}(N,N)h)=0.
\]
Therefore $\Pi(L(h))=L(h)$ and $h$ is a first eigenfunction of $L$, and thus $h=0$, since $h\in V^{2,\alpha}$ is orthogonal to $\phi_1$. Moreover, for $(f,g)\in V^{0,\alpha}\times C^{1,\alpha}(\partial\Sigma)$, one can find a solution $h\in V^{2,\alpha}$ such that $D_vF(0,0)(h)=(f,g)$. Indeed, let us first assume that $g\in C^{2,\alpha}(\partial\Sigma)$ and consider $\beta$ a function in $C^{2,\alpha}(\Sigma)$ such that $\partial_\nu \beta-\ii_{\partial M}(N,N)\beta=g$ (for example, $\beta=0$ and $\partial_\nu\beta=g$ on $\partial\Sigma$). By adding a multiple of $\phi_1$, we may assume that $\beta$ is orthogonal to $\phi_1$. The function $\tilde f=f-\lambda_1\beta+\Pi(L\beta)$ is then in $V^{0,\alpha}$. Thus the Fredholm alternative ensures the existence of a solution $u\in V^{2,\alpha}$ to the system
\[
\begin{cases}
-Lu+\lambda_1 u=\tilde f&\text{on }\Sigma,\\
\partial_\nu u-\ii_{\partial M}(N,N)u=0&\text{on }\p\Sigma.
\end{cases}
\]
The regularity of the solution comes from results given in \cite{Grs}. Then the function $u+\beta$ solves $D_vF(0,0)(u+\beta)=(f,g)$. When $g\in C^{1,\alpha}(\p\Sigma)$, we find the solution by approximating $g$ by functions in $C^{2,\alpha}(\p\Sigma)$ and using Schauder type estimates.

Hence the differential is invertible and there is a family $(v_t)_{t\in(-\eps,\eps)}$ such that $F(t,v_t)=0$ and $v_0=0$. In particular, $\partial_t{v_t}_{|t=0}=0$. Then $u_t=t\phi_1+v_t$ will satisfy the Lemma. In fact, the first item is satisfied and we notice that the graph of $u_t$ has free boundary, thanks to the second coordinate of $F$, so the last item is also satisfied. Besides,
\[
\partial_t H(u_t)_{|t=0}=-L\phi_1=-\lambda_1\phi_1.
\]
Then, since $\lambda_1<0$ and $\phi_1>0$, $H(u_t)$ has the expected sign. 
\end{proof}

\subsection{The area of free boundary minimal hypersurfaces}\label{subsec:areaofFBMH}

In the sequel, we will obtain an estimate for the perimeter of a free 
boundary 
minimal surface. For many reasons, it would be interesting to obtain an area 
estimate. In fact, when the sectional curvature is non-negative, we have such 
an estimate thanks to the following result:

\begin{prop}\label{prop:area}
Let $(M,g)$ be a Riemannian $n$-manifold with smooth boundary. Assume that 
$(M,g)$ has non-negative sectional curvature and $\ii_{\partial M}\ge 
g_{|\partial M}$. If $\Sigma$ is a compact free boundary minimal 
hypersurface in 
$(M,g)$, then
\[
(n-1)|\Sigma|\le|\partial\Sigma|.
\]
\end{prop}

\begin{proof}
Let $q$ be a point in $\partial M$ and $U$ a neighborhood of $q$ in $\partial 
M$. Consider $F:U\times [0,T)\to M$ the map defined by 
$F(x,t)=\exp_x(t\eta(x))$, where $\eta(x)$ is the inward unit normal to 
$\partial M$. If $s$ is less than the first focal time along the geodesic 
$t\mapsto F(q,t)$, then $F$ is a local diffeomorphism from a neighborhood of 
$(q,s)$ to a neighborhood of $p=F(q,s)$. On this neighborhood, we can then 
define a ``distance'' function $d_q$ by $d_q(F(x,t))=t$. Moreover, if 
$d_q(m)=t$, its Hessian $\barre\nabla^2d_q$ at $m$ is given by 
$-\ii_{\{d_q=t\}}$, where the second fundamental form is computed with respect 
to the unit normal $\barre\nabla d_q$. By comparison with Euclidean space (see 
H.~Karcher~\cite{Kar2}), the curvature assumptions give that 
$\ii_{\{d_q=t\}}\ge\frac1{1-t}g_{|\{d_q=t\}}$. This ensures that $d_q$, and 
$s$, are bounded from above by $1$ and, for $\bar f_q=d_q^2-2d_q$, 
$\barre\nabla^2\bar f_q\ge2g$.

Let $d$ be the distance function to $\partial M$. If $d(p)$ is realized by a 
geodesic between $p$ and $q\in\partial M$ such that $d_q$ is defined near $p$, 
then $d\le d_q\le 1$ and $d(p)=d_q(p)$; notice that $d\le 1$ on all of $M$. Let 
us define $\bar f=d^2-2d$. Let $\Sigma$ be a compact free boundary 
minimal hypersurface in $M$ and consider~$f$ the restriction of $\bar f$ to 
$\Sigma$. We are going to prove the following:
\begin{claim*}
In the viscosity sense, $\Delta f\ge 2(n-1)$ (see \cite{Bar,CrIsLi}).
\end{claim*}

\begin{proof}[Proof of the claim]
For $p\in\Sigma$, let $\phi$ be a smooth function on a neighborhood of $p$ in $\Sigma$ such that $\phi\ge f$ and $\phi(p)=f(p)$. Let $\gamma$ be a unit speed geodesic with $\gamma(0)=q\in\partial M$ and $\gamma(s)=p$ that realizes the distance to $\p M$, \textit{i.e.} $d(p)=s$. If $s$ is less than the first focal time along~$\gamma$, then $d\le d_q\le 1$ and $d(p)= d_q(p)$. Therefore $\phi\ge f\ge d_q^2-2d_q$ with equality at $p$. Hence $\Delta \phi(p)\ge \Delta(d_q^2-2d_q)(p)\ge 2(n-1)$, since $\Sigma$ is minimal.

If $s$ is the first focal time along $\gamma$, we will modify the metric $g$ to 
obtain the result. To do so, we will construct a function $u$ on $M$ with 
compact support in a neighborhood of $\gamma(s/2)$ that does not contain $p$ 
such that $u=0$ and $\barre\nabla u=0$ along $\gamma$ and 
$\barre\nabla^2u(\gamma(t))=\alpha(t)\pi^*g$ for some non-negative function 
$\alpha(t)$ with $\alpha(s/2)>0$, where $\pi$ is the orthonormal projection 
onto the normal bundle of $\gamma$ and $\pi^*g(a,b)=g(\pi(a),\pi(b))$. In fact, 
consider a local chart near $\gamma(s/2)$ with coordinates $(x_1,\ldots,x_n)$ 
in the open $n$-cube 
$\boC_\delta=(s/2-\delta,s/2+\delta)\times(-\delta,\delta)\times\cdots\times(-\delta,\delta)$
 with $0<\delta<s/2$ such that, in these coordinates, $\gamma$ is the curve 
$t\mapsto(t,0,\ldots,0)$ and the Riemannian metric $g$ satisfies 
$g_{ij}=\delta_{ij}$ along $\gamma$. Let $\alpha(t)$ be a non-negative function 
with compact support in $(s/2-\delta/2,s/2+\delta/2)$ such that 
$\alpha(s/2)>0$. Then we can define $u$ by
\[
u(x_1,\ldots,x_n)=\alpha(x_1)\varphi(x_2,\ldots,x_n)\frac{x_2^2+\cdots+x_n^2}{2}
\]
in the coordinate neighborhood and $u=0$ outside it, where $\varphi$ is a 
non-negative function with compact support in $(-\delta,\delta)^{n-1}$ such 
that $\varphi=1$ on $(-\delta/2,\delta/2)^{n-1}$. Notice that $u$ has compact 
support in $\boC_\delta$. Now, let $g_\eps=e^{2\eps u}g$ for $\eps>0$. Along 
$\gamma$, the metrics $g$ and $g_\eps$ coincide, $\gamma$ is a geodesic for 
$g_\eps$ and the covariant derivatives $\frac{D}{dt}$ and $\frac{D_\eps}{dt}$ 
are the same. Moreover, the curvature tensor associated with $g_\eps$ satisfies
\[
R_\eps(\gamma'(t),X)\gamma'(t)=R(\gamma'(t),X)\gamma'(t)-\eps\alpha(t)X
\]
for any vector field $X$ orthogonal to $\gamma$ (see~\cite{Bes}).

For $0<a\le s$, let $\boV_a$ be the space of continuous and piecewise differentiable vector fields $V$ along $\gamma_{|[0,a]}$ such that $V\perp\gamma'$ and $V(a)=0$. For $V,W\in \boV_a$, let us define
\[
I_a^\eps(V,W)=\int_0^a(V',W')-(R_\eps(\gamma',V)\gamma',W)dt.
\]
By the index theorem (see~\cite{DoC}), the index of $I_a^\eps$ is given by the number of $t\in(0,a)$ that is a focal time of $\gamma$, for the metric $g_\eps$, each one counted with multiplicity. Notice that 
\[
I_a^\eps(V,V)=I_a^0(V,V)+\eps\int_0^a\alpha(t)|V|^2dt\ge I_a^0(V,V).
\]
Therefore, the index of $I_a^\eps$ is less than or equal to the index of 
$I_a^0$. On the other hand, by assumption, the first focal time along $\gamma$ 
is $s$, so $I_a^0$ is non-negative for $a\le s$ (no time $t<s$ is a focal 
time). If $s$ is a focal time of $\gamma$ for the metric $g_\eps$, there is a 
nonzero Jacobi field $J$ along $\gamma$ with $J'(0)=0$ and $J(s)=0$. Thus
\[
0=I_s^\eps(J,J)=I_s^0(J,J)+\eps\int_0^{s}\alpha(t)|J|^2dt\ge\eps\int_0^{s}\alpha(t)|J|^2dt>0,
\]
which is a contradiction. So the first focal time along $\gamma$ for the metric $g_\eps$ is strictly larger than~$s$ for $\eps>0$.

We notice that, for the metric $g_\eps$, the distance function $d^\eps$ to $\partial M$ satisfies $d^\eps\ge d$ and $d^\eps(p)=d(p)$, since we keep the length of $\gamma$. Bearing in mind that $s$ is not a focal time for $g_\eps$, there is a local distance function $d_q^\eps$ defined near $p$ such that $d^\eps\le d_q^\eps$ and $d^\eps(p)=d_q^\eps(p)$. The metric $g_\eps$ has sectional curvature bounded from below by some $-c_\eps^2$ with $c_\eps>0$ and $c_\eps\to0$ as~$\eps\to0$. By comparison with the hyperbolic space~\cite{Kar2}, we have 
\[
\barre\nabla^2d_q^\eps(p)\le-c_\eps\frac{(1+c_\eps)+(1-c_\eps)e^{2c_\eps s}}{(1+c_\eps)-(1-c_\eps)e^{2c_\eps s}}\,g_{|\{d_q^\eps=s\}};
\]
notice that, near $p$, $g_\eps=g$. We then have $\phi\ge f\ge(d_q^\eps)^2-2d_q^\eps$ with equality at $p$. Hence
\[
\Delta\phi(p)\ge\Delta\big((d_q^\eps)^2-2d_q^\eps\big)\ge2(n-1)\min\left(1,(1-s)c_\eps\frac{(1+c_\eps)+(1-c_\eps)e^{2c_\eps s}}{(1+c_\eps)-(1-c_\eps)e^{2c_\eps s}}\right),
\]
where above we have used that $\Sigma$ is minimal at $p$ for $g_\eps$. Finally, letting $\eps\to0$, we get $\Delta\phi(p)\ge2(n-1)$.
\end{proof}

The function $f$ is smooth near $\partial\Sigma$ and $\partial_\nu f=2$ on $\partial\Sigma$. So, if $f$ was smooth, integrating $\Delta f\ge2(n-1)$ on $\Sigma$ and applying Stokes formula would give the expected result. Since \textit{a priori} $f$ is not smooth, we proceed as follows. Let $c=\frac{|\partial\Sigma|}{|\Sigma|}$ and assume that $c<n-1$. Then there is a smooth function $u$ on $\Sigma$ solving the problem:
\[
\begin{cases}
\Delta u=2c&\text{on }\Sigma,\\
\partial_\nu u=2&\text{on }\partial\Sigma.
\end{cases}
\]
Let us consider $m_0=\min_{\partial\Sigma}u$. Thus $u-m_0\ge 0=f$ on $\partial \Sigma$. Since $c<n-1$, the maximum principle for viscosity subsolutions~\cite[Theorem~3.3]{CrIsLi} gives that $u-m_0\ge f$ on $\Sigma$. Now, let $p\in \partial \Sigma$ be such that $u(p)=m_0$. Near $p$, both $u$ and $f$ are smooth, and we have $u-m_0\ge f$, $u(p)-m_0=f(p)$, $\partial_\nu(u-m_0)(p)=2=\partial_\nu f(p)$ and $\Delta(u-m_0)(p)=2c<2(n-1)\le \Delta f(p)$, which is impossible. So we have proved that $c\ge(n-1)$, that is, $(n-1)|\Sigma|\le|\partial\Sigma|$.
\end{proof}

\begin{rmk}
We state the above result for hypersurfaces. However, we can remark that the 
proof works in any codimension: for any minimal $k$-submanifold $\Sigma$ with 
free boundary in $M$, one has $k|\Sigma|\le|\partial\Sigma|$.
\end{rmk}

\section{Min-max theorems}
\label{sec:minmax}
In this section, we recall some min-max constructions of free boundary minimal surfaces in Riemannian $3$-manifolds with convex boundary. These constructions first appeared in the work of M.~Grüter and J.~Jost~\cite{GrJo1} for free boundary disks in convex domains of $\R^3$ and then they were extended to manifolds with convex boundary by J.~Jost~\cite{Jos2}. Here we give a presentation of these constructions as they appear in the work of T.H.~Colding and C.~De Lellis~\cite{CoDeL} for the no boundary case and in the work of M.M.-C.~Li~\cite{Li_m} for the boundary case.

We study two different cases: the first one applies to compact $3$-manifolds with smooth boundary; the second deals with compact $3$-manifolds with piecewise smooth boundary and we want to prevent the minimal surface to attach to a certain part of the boundary.

\subsection{The unconstrained case}\label{subsec:unconstrained}

Let $M$ be a compact Riemannian $3$-manifold with smooth boundary and $I=[a,b]$. A family $\{\Sigma_t\}_{t\in I}$ of closed subsets of $M$ is a \textit{generalized smooth family of surfaces}, or a \textit{sweepout}, if there are finite subsets $T\subset I$ and $P\subset M$ such that:
\begin{itemize}
\item[$(a)$] $\boH^2(\Sigma_t)$ is a continuous function of $t$;
\item[$(b)$] $\Sigma_t\to\Sigma_{t_0}$ in the Hausdorff topology whenever $t\to t_0$;
\item[$(c)$] For every $t\in I\setminus T$, $\Sigma_t$ is a compact surface in $M$ with (possibly empty) boundary such that $\partial\Sigma_t=\Sigma_t\cap\partial M$ and this intersection is transverse;
\item[$(d)$] For $t\in T$, $\Sigma_t\setminus P$ is a surface in $M\setminus P$ with boundary satisfying $\partial(\Sigma_t\setminus P)=(\Sigma_t\setminus P)\cap\partial M$ and this intersection is transverse;
\item[$(e)$] $\Sigma_t$ varies smoothly in $I\setminus T$;
\item[$(f)$] For $\tau\in T$, $\Sigma_t\to\Sigma_\tau$ smoothly in $M\setminus P$ as $t\to\tau$. 
\end{itemize}

Let $\Diff_0$ be the set of diffeomorphisms of $M$ that are isotopic to the identity. If $\{\Sigma_t\}_{t\in I}$ is a generalized smooth family of surfaces in $M$ and $\psi:I\times M\to M$ is a smooth map such that $\psi_t=\psi(t,\cdot)\in\Diff_0$, one can define a new generalized smooth family of surfaces by $\{\psi_t(\Sigma_t)\}_{t\in I}$. A set $\Lambda$ of generalized smooth families of surfaces is said to be \textit{saturated} if it is closed under the above operation and the cardinal of the set of singular points $P$ is uniformly bounded among the elements of $\Lambda$.

One can define the maximal area of a generalized smooth family of surfaces $\{\Sigma_t\}_{t\in I}$ by
\[
\mathcal{F}(\{\Sigma_t\})=\max_{t\in I}\boH^2(\Sigma_t).
\]
Then, if $\Lambda$ is a saturated set of generalized smooth families of surfaces, one defines its \textit{width} with respect to the Riemannian metric $g$ on $M$ by
\[
W(\Lambda,g)=\inf_{\sigma\in \Lambda}\mathcal{F}(\sigma).
\]

\begin{exam}
In this paper, we will only consider the case where $M$ is the unit $3$-ball
\[
B=\{(x,y,z)\in\R^3\mid x^2+y^2+z^2\le 1\}.
\]
One generalized smooth family of surfaces is $\{\Gamma_t\}$ given by
\[
\Gamma_t=B\cap\{z=t\},\text{ for }t\in [-1,1].
\]
Besides, $\Lambda_0$ will denote the smallest saturated set of sweepouts containing $\{\Gamma_t\}$. We define $W(B,g)=W(\Lambda_0,g)$ for any Riemannian metric $g$ on $B$.
\end{exam}

\begin{thm}\label{th:min_max}
Assume that $M$ is endowed with a Riemannian metric $g$ such that $\ii_{\partial M}$ is positive definite. Let $\Lambda$ be a saturated set of generalized smooth families of surfaces with $W(\Lambda,g)>0$. Then there exist a collection $\{\Gamma_i\}_{1\le i\le N}$ of free boundary minimal surfaces in $M$ and a set $\{n_i\}_{1\le i\le N}$ of positive integers such that
\[
W(\Lambda,g)=\sum_{i=1}^N n_i|\Gamma_i|.
\]
\end{thm}

Let us briefly explain the proof of the above result. Let $\{\Sigma_t^n\}$ be a \textit{minimizing sequence} in $\Lambda$: $\lim_n\boF(\{\Sigma_t^n\})=W(\Lambda,g)$. If $(t_n)$ is a sequence such that $\lim_n\boH^2(\Sigma_{t_n}^n)=W(\Lambda,g)$, we say that $(\Sigma_{t_n}^{n})$ is a \textit{min-max sequence}. The basic idea of the proof consists in showing that the varifold limit of a min-max sequence takes the form $\sum_{i=1}^Nn_i\vbf(\Gamma_i)$ for some free boundary minimal surfaces $\Gamma_i$.

The first step of the proof, the pull-tight procedure \cite[Proposition~5.1]{Li_m}, allows to assume that the minimizing sequence $\{\Sigma_t^n\}$ is such that any varifold limit of a min-max sequence $(\Sigma_{t_n}^{n})$ is a freely stationary varifold.

The second step of the proof consists in selecting a min-max sequence $\Sigma^j=\Sigma_{t_j}^{n_j}$ such that $\Sigma^j$ is $1/j$-almost minimizing in any small annuli (see Definition~3.2 in \cite{CoDeL} or Definition~3.6 in \cite{Li_m} replacing $\mathfrak{Is}_{\out}$ by $\mathfrak{Is}_{\tan}$) and such that, in any small annuli, $\Sigma^j$ is smooth when $j$ is large. The proof of this can be found in \cite[Proposition~5.1]{CoDeL} or \cite[Proposition~6.4]{Li_m}.

The last step is to prove the regularity of the support of the varifold limit $V=\lim\vbf(\Sigma^j)$. For the points in the interior of $M$, one can take a look at Sections 6 and 7 of \cite{CoDeL}, and for those on the boundary of $M$, at the work of Jost~\cite{Jos2}.

\subsection{The constrained case}\label{subsec:constrained}

Now, let us consider the case where $M$ is a compact Riemannian $3$-manifold with piecewise smooth boundary whose boundary is the union $\partial M=\partial_0M\cup \partial_+M$ of two smooth pieces: $\partial_0M$ and $\partial_+M$, each one being a union of closures of connected components of $\partial M\setminus\boC(M)$, such that $\partial_0M\cap\partial_+M=\boC(M)$. Our aim is to allow constructions similar to those in the unconstrained case that prevent the free boundary to intersect $\partial_0M$.

In this situation, we consider $I=[0,1]$ and a family $\{\Sigma_t\}_{t\in I}$ of closed subsets of $M$ is a \textit{constrained generalized smooth family of surfaces}, or a \textit{constrained sweepout}, if items $(a)$, $(b)$, $(e)$ and $(f)$ are satisfied and the others are replaced by
\begin{itemize}
\item[$(c')$] For every $t\in I\setminus(T\cup\{0\})$, $\Sigma_t$ is a compact surface in $M\setminus\p_0M$ with (possibly empty) boundary satisfying $\partial\Sigma_t=\Sigma_t\cap\partial_+M$ and this intersection is transverse;
\item[$(d')$] For $t\in T$, $\Sigma_t\subset M\setminus\partial_0M$ and $\Sigma_t\setminus P$ is a surface in $M\setminus P$ with boundary such that $\partial(\Sigma_t\setminus P)=(\Sigma_t\setminus P)\cap\partial_+M$ and this intersection is transverse. Besides, $\Sigma_0=\p_0M$.
\end{itemize}

In this context, $\Diff_0$ will denote the set of diffeomorphisms of $M$ that are isotopic to the identity and equal to the identity in a neighborhood of $\partial_0M$. We can then consider saturated sets $\Lambda$ of generalized smooth families of surfaces as in the unconstrained case.

If $g$ is the Riemannian metric on $M$, we can then define 
$\mathcal{F}(\sigma)$ for $\sigma\in\Lambda$ and $W(\Lambda,g)$ as in the 
unconstrained case. 

\begin{defn}\label{def:freemcf}
We say that $M$ has a \emph{free boundary mean-convex foliation near} 
$\partial_0M$ if 
there exists a non-negative continuous function $f$ defined on an open 
neighborhood of 
$\partial_0M$ such that $\partial_0M=f^{-1}(0)$ and, outside $\boC(M)$, $f$ is 
a smooth submersion such that $\partial_\nu f=0$ along $\partial_+M\setminus 
\boC(M)$ and 
$f^{-1}(t)$ has positive mean curvature in the 
$\nabla f$ direction for small $t$.
\end{defn}
Notice that the property $\partial_\nu f=0$ ensures that $f^{-1}(t)$ is a free 
boundary hypersurface. In Appendix~\ref{sec:foliat}, we prove that we can 
construct such a foliation under some reasonable assumptions on the geometry of 
$M$.

\begin{thm}\label{th:constraint_min_max}
Assume that $M$ is endowed with a Riemannian metric $g$ such that 
$\ii_{\partial_+M}$ is positive definite and $(M,g)$ has a free boundary 
mean-convex foliation near $\partial_0M$. Let $\Lambda$ be a saturated set of 
generalized smooth families of surfaces satisfying $W(\Lambda,g)>|\partial_0 
M|$. Then there exist a collection $\{\Gamma_i\}_{1\le i\le N}$ of minimal 
surfaces in $M$ with free boundary in $\partial_+M\setminus\partial_0M$ and a 
set $\{n_i\}_{1\le i\le N}$ of positive integers such that
\[
W(\Lambda,g)=\sum_{i=1}^N n_i|\Gamma_i|.
\]
\end{thm}

\begin{exam}\label{ex:lambda+}
In this paper, we will consider $M$ the unit half-ball $B^+=B\cap\{z\ge0\}$ with $\partial_0B^+=B\cap\{z=0\}$ and $\partial_+B^+=\partial B\cap\{z\ge0\}$. One constrained generalized smooth family of surfaces is $\{\Gamma_t^+\}$ given by
\[
\Gamma_t^+=B^+\cap\{z=t\},\text{ for }t\in[0,1].
\]
We denote by $\Lambda_0^+$ the smallest saturated set of constrained sweepouts containing $\{\Gamma_t^+\}$.
\end{exam}

The proof of Theorem~\ref{th:constraint_min_max} is based on similar ideas to those presented in the proof of \cite[Theorem~2.1]{MaNe3} by Marques and Neves.

In order to prove the result, we first construct a minimizing sequence $\{\Sigma_t^n\}$ satisfying: there are $a,\delta>0$ such that
\begin{equation}\label{eq:const_minmax}
\forall n,t,\quad\boH^2(\Sigma_t^n)\ge W(\Lambda,g)-\delta\implies d(\Sigma_t^n,\partial_0M)\ge a.
\end{equation}

To do so, let $f$ be a function given by Definition~\ref{def:freemcf} 
and define $S_t=f^{-1}(t)$. If $f$ is defined on the open neighborhood $\Ome$, 
there is $\eps>0$ such that $f> 2\eps$ on $\partial\Ome\setminus\partial M$. We 
can change the definition of $f$ outside $\Ome'=\{f<2\eps\}$ to extend $f$ 
to a non-negative function on $M$ with $f\ge 2\eps$ outside $\Ome'$.
Then fix 
$0<a<\frac12d(S_\eps,\p_0M)$ and 
$0<\delta<\frac12(W(\Lambda,g)-|\partial_0M|)$. We will prove that such $a$ and 
$\delta$ are convenient.

Let us consider a minimizing sequence $\{\Sigma_t^n\}$ in $\Lambda$. Fix $n$. 
Because of items $(a)$, $(b)$, $(c')$, and $(d')$, there is $b>0$ such that
\[
\boH^2(\Sigma_t^n)\ge W(\Lambda,g)-\delta \implies d(\Sigma_t^n,\partial_0M)\ge b.
\]
We may assume that $b<a$; if not, we do nothing. We have the following lemma:

\begin{lem}\label{lem:vectorfield}
With the above notations, if $0<\mu_1<\mu_2<\eps$, there is a vector field $X$ on $M$ which vanishes near $\partial_0M$, is tangent to $\partial_+M$, and whose associated flow $(\Phi_s)$ satisfies the following:
\begin{itemize}
\item for any $t\in[0,1]$, if $\{\Sigma_t^n\}\subset\{f\ge\mu_1/3\}$, $\boH^2(\Phi_s(\Sigma_t^n))\le\boH^2(\Sigma_t^n)$ for each $s\ge 0$;
\item for any $p\in\{f\ge\mu_1\}$, $\Phi_1(p)\in \{f\ge\mu_2\}$.
\end{itemize}
\end{lem}

We postpone the proof of the above Lemma and finish the construction of our 
sweepout. Let $0<\mu_1 <\mu_2<\eps$ be such that $S_{\mu_1}$ is in the 
$b$-tubular neighborhood of $\partial_0M$ and $S_{\mu_2}$ is outside the 
$a$-tubular neighborhood of $\partial_0 M$. Let $X$ and $(\Phi_s)$ be given by 
the above Lemma. Because of item~$(b)$, the function 
$t\mapsto\inf_{\Sigma_t^n}f$ is continuous and non-negative. Therefore, we can 
consider a smooth function $\theta:[0,1]\to [0,1]$ such that 
$\inf_{\Sigma_t^n}f\le\theta(t)\le\inf_{\Sigma_t^n}f+\mu_1/3$. Let 
$T:[0,1]\to[0,1]$ be a smooth function satisfying $T(x)=0$ if $x\le 2\mu_1/3$ 
and $T(x)=1$ if $x\ge\mu_1$. Let us then consider the generalized smooth family 
of surfaces $\{\Phi_{T(\theta(t))}(\Sigma_t^n)\}$. It belongs to $\Lambda$. Let 
$t\in[0,1]$. If $\Sigma_t^n\cap\{f\le \mu_1/3\}\neq\emptyset$, 
$\theta(t)\le2\mu_1/3$, and then $\Phi_{T(\theta(t))}(\Sigma_t^n)=\Sigma_t^n$. 
If $\Sigma_t^n\cap \{f\le \mu_1/3\}= \emptyset$, then 
$\boH^2(\Phi_{T(\theta(t))}(\Sigma_t^n))\le\boH^2(\Sigma_t^n)$. Thus we see 
that $\boF(\{\Phi_{T(\theta(t))}(\Sigma_t^n)\})\le\boF(\{\Sigma_t^n\})$. 
Besides, if $\boH^2(\Phi_{T(\theta(t))}(\Sigma_t^n))\ge W(\Lambda,g)-\delta$, 
then $\boH^2(\Sigma_t^n)\ge W(\Lambda,g)-\delta$. In particular, 
$\Sigma_t^n\subset\{f\ge\mu_1\}$, $\theta(t)\ge\mu_1$, and then 
$\Phi_{T(\theta(t))}(\Sigma_t^n)=\Phi_1(\Sigma_t^n)\subset\{f\ge\mu_2\}$. Thus 
$d(\Phi_{T(\theta(t))}(\Sigma_t^n),\partial_0M)\ge a$. Therefore 
$\{\Phi_{T(\theta(t))}(\Sigma_t^n)\}$ is a minimizing sequence which satisfies 
\eqref{eq:const_minmax} (notice that $X$ depends on $b$ and thus on $n$).

\begin{proof}[Proof of Lemma~\ref{lem:vectorfield}]
Let $c>0$ be chosen latter. Let $\phi$ be a smooth non-negative function 
defined 
on $\R_+$ which is positive on $[0,\mu_2]$, vanishes on a neighborhood 
$[\eps,+\infty)$, and satisfies $\phi'+c\phi\le 0$. Let $\alpha$ be a smooth 
non-negative function on $\R_+$ vanishing near $0$ and such that $\alpha=1$ on 
$[\mu_1/3,+\infty]$. Let $X$ be the vector field defined by 
$X=\phi(f)\alpha(f)N$, where $N=\frac{\nabla f}{|\nabla f|}$ (extended 
by zero 
outside of 
$\Ome$). We notice 
that this vector field vanishes close to $\partial_0M$ and is tangent to 
$\partial_+ M$. 

Let us fix a $2$-plane $V=\Span(e_1,e_2)$ in $T_p M$ where $p\in\{f\ge\mu_1/3\}$ and compute $\Div_VX$. We may assume that $p\in\{f<\eps\}$. First we notice that $\alpha(f)=1$ at $p$ and we may assume $e_1\in T_pS_{f(p)}$. Let $g\in T_pS_{f(p)}$ be such that $(e_1,g,N)$ is an orthonormal basis of $T_pM$. Then 
\begin{align*}
\Div_VX&=(\nabla_{e_1}\phi(f)N,e_1)+(\nabla_{e_2}\phi(f)N,e_2)\\
&=\phi(f)(\nabla_{e_1} N,e_1)+\phi'(f)(e_2,\nabla f)(N,e_2)+\phi(f)(\nabla_{e_2}N,e_2).
\end{align*}
Since $e_2=(e_2,g)g+(e_2,N)N$, we have
\begin{align*}
(\nabla_{e_2} N,e_2)&=(e_2,g)(e_2,N)(\nabla_N N,g)+(e_2,g)^2(\nabla_gN,g)\\
&=(e_2,g)(e_2,N)(\nabla_N N,g)+(1-(e_2,N)^2)(\nabla_g N,g).
\end{align*}
So, if $H_0>0$ is a lower bound on the mean curvature of $S_t$ and  
keeping in mind that $\phi$ is non-negative and $\phi'\le-c\phi\le0$, we have
\begin{align*}
\Div_VX&=-\phi(f)H_{S_{f(p)}}(p)+\phi'(f)(e_2,N)^2(\nabla f,N)\\
&\qquad+\phi(f)\left((e_2,g)(e_2,N)(\nabla_N N,g)-(e_2,N)^2(\nabla_gN,g)\right)\\
&\le-\phi(f)H_0+\phi'(f)(e_2,N)^2 
k+\frac{\eps}{2}\phi(f)+\frac1{2\eps}(e_2,N)^2 
K^2\phi(f)+(e_2,N)^2 K\phi(f)\\
&\le\phi(f)\left(\frac{\eta}{2}-H_0\right)+(e_2,N)^2\left(k\phi'(f)+ 
(K+\frac{K^2}{2\eta})\phi(f)\right)
\end{align*}
for any $\eta>0$, where $k>0$ is a lower bound on $(\nabla 
f,N)$ and $K$ is an upper bound on the norm of the operator $\nabla_\cdot N$. 
So choosing $\eta<H_0$ and 
$c>\frac1k(K+\frac{K^2}{2\eta})$ ensures that $\Div_VX\le 0$. 
Notice that, along the flow associated to $X$, the function $f$ is increasing 
and $X$ does not vanish on $\{\mu_1/3\le f\le \mu_2\}$. Thus, multiplying $X$ 
by some positive constant, we can ensure the second item of the Lemma. Finally, 
since $\Div_V X\le 0$ for any $2$-plane $V$ on $\{f\ge\mu_1/3\}$, the first 
item is true by the first variation of area formula.
\end{proof}

Once the construction of a minimizing sequence $\{\Sigma_t^n\}$ satisfying 
\eqref{eq:const_minmax} is done, the rest of the proof is similar to the proof 
of \cite[Theorem~2.1]{MaNe3} by Marques and Neves. The pull-tight procedure and 
the selection of a sequence that is almost-minimizing in any small annuli can 
be done away from $\partial_0M$. This ensures the existence of a min-max 
sequence $\Sigma_{t_{j}}^{n_{j}}$ that stays away from $\partial_0M$ and then 
the regularity of its varifold limit. 

\section{\textcolor{black}{Topological control}}\label{sec:topologicalcontrol}

In this section, we explain the topological control we have on the free 
boundary minimal surfaces produced by the above constructions. The result is a 
consequence of the work of De~Lellis and Pellandini~\cite{DeLPe} (see 
also the recent result in \cite{FrSc}).

We denote by $\boO$ (resp. $\boO_\partial$) the set of compact orientable 
connected surfaces without boundary (resp. with nonempty boundary). Let $\boN$ 
be the set of compact nonorientable connected surfaces with (possibly empty) 
boundary. An element $\Gamma$ of $\boO$ or $\boO_\partial$ is homeomorphic to 
the connected sum of the sphere $\S^2$ and $\gbf(\Gamma)$ tori $\T^2$ to which 
$\kbf(\Gamma)$ disks are removed. An element $\Gamma$ of $\boN$ is homeomorphic 
to the connected sum of $\gbf(\Gamma)$ projective planes $ P^2$ to which 
$\kbf(\Gamma)$ disks are removed (see~\cite{GaXu}). $\gbf(\Gamma)$ is called 
the \textit{genus} of the surface of $\Gamma$.

For a surface $\Sigma$, whose connected components are the $\Gamma_i$, we define
\[
\tbf(\Sigma)=\sum_{\Gamma_i\in 
\boO}2\gbf(\Gamma_i)+\sum_{\Gamma_i\in\boO_\partial}\big(2\gbf(\Gamma_i)+ 
\kbf(\Gamma_i)-1\big)+\sum_{\Gamma_i\in\boN}\big(\gbf(\Gamma_i)+\kbf(\Gamma_i)-1\big).
\]

From the above constructions, we recall that at the second step we obtain a min-max sequence $(\Sigma_{t_j}^j)$ such that $\Sigma_{t_j}^j$ is $1/j$-almost minimizing in any small annuli and, in any small annuli, $\Sigma_{t_j}^j$ is smooth. Moreover, as varifolds, $\vbf(\Sigma_{t_j}^j)\to\sum_{i=1}^Nn_i\vbf(\Gamma_i)$, where the $\Gamma_i$ are free boundary minimal surfaces.

We can now state our control on the topology (compare 
with~\cite[Theorem~1.6]{DeLPe}).

\begin{thm}\label{th:top}
Assume that any surface $\Sigma_t$ appearing in $\Lambda$ is orientable for $t\in I\setminus T$. Let $(\Sigma_{t_j}^j)$ be a min-max sequence as above. Then 
\[
\tbf(\bigcup_{i=1}^n\Gamma_i)\le\liminf_{j\to\infty}\liminf_{t\to t_j}\tbf(\Sigma_{t}^j).
\]
\end{thm}

Let us explain how one can adapt the arguments in \cite{DeLPe}. First, let us explain a type of surgery that we will use for orientable surfaces. If $\Sigma$ is a compact surface with boundary, an \textit{open strip} in $\Sigma$ is an open subset $S\subset\Sigma$ that is homeomorphic to $(0,1)\times[0,1]$. Notice that the homeomorphism sends $\partial\Sigma\cap S$ to $(0,1)\times\{0,1\}$, thus an open strip joins two sub-arcs of $\partial\Sigma$. An \textit{open half-disk} in $\Sigma$ is an open subset $D\subset\Sigma$ that is homeomorphic to 
\[
\{(x,y)\in\R^2\mid x^2+y^2<1\text{ and }y\ge0\}.
\]
Here the homeomorphism sends $\partial\Sigma\cap D$ to $(-1,1)\times\{0\}$.

\begin{defn}
Let $\Sigma$ and $\Sigma'$ be two compact oriented surfaces. We say that \textit{$\Sigma'$ is obtained from~$\Sigma$ by cutting away a neck} if 
\begin{itemize}
\item either $\Sigma\setminus\Sigma'$ is homeomorphic to $\S^1\times (0,1)$ and $\Sigma'\setminus\Sigma$ is homeomorphic to the disjoint union of two open disks,
\item or $\Sigma\setminus\Sigma'$ is an open strip in $\Sigma$ and $\Sigma'\setminus\Sigma$ is the disjoint union of two open half-disks in~$\Sigma'$.
\end{itemize}
We say that \textit{$\widetilde\Sigma$ is obtained from $\Sigma$ through surgery} if there is a finite number of surfaces
\[
\Sigma_0=\Sigma,\Sigma_1,\dots,\Sigma_{N}=\widetilde\Sigma
\]
such that, for any $0\le k\le N-1$,
\begin{itemize}
\item either $\Sigma_{k+1}$ is homeomorphic to $\Sigma_k$,
\item or $\Sigma_{k+1}$ is obtained from $\Sigma$ by cutting away a neck.
\end{itemize}
\end{defn}

Assume that $\Sigma'$ is obtained from $\Sigma$ by cutting away a neck and let 
us compare $\tbf(\Sigma)$ with~$\tbf(\Sigma')$. If we are in the first item, we 
have two possibilities. The first one is when the annulus does not separate 
$\Sigma$; then, in this case, $\Sigma'$ and $\Sigma$ have the same number of 
connected components and exactly one of them sees its genus decreased 
by $1$ so 
that $\tbf(\Sigma')=\tbf(\Sigma)-2$. If the annulus separates~$\Sigma$, then 
one connected component $\Gamma$ of $\Sigma$ gives birth to two connected 
components $\Gamma_1'$ and $\Gamma_2'$ of $\Sigma'$ with 
$\gbf(\Gamma_1')+\gbf(\Gamma_2')=\gbf(\Gamma)$ and 
$\kbf(\Gamma_1')+\kbf(\Gamma_2')=\kbf(\Gamma)$. Depending on whether 
$\kbf(\Gamma_i')=0$ or not, we have $\tbf(\Sigma')=\tbf(\Sigma)$ or 
$\tbf(\Sigma)-1$. If we are in the second item, we also have several cases to 
consider. First, the strip can join two different connected components 
of~$\partial\Sigma$. In this case, $\Sigma'$ and $\Sigma$ have the same number 
of connected components and exactly one of them sees $\kbf$ decreased 
by $1$ so 
that $\tbf(\Sigma')=\tbf(\Sigma)-1$. If the strip joins two sub-arcs of the 
same connected component of $\partial\Sigma$, then the strip may separate 
$\Sigma$ or not. In the first case, one connected component $\Gamma$ of 
$\Sigma$ gives birth to two connected components $\Gamma_1'$ and $\Gamma_2'$ of 
$\Sigma'$ with $\gbf(\Gamma_1')+\gbf(\Gamma_2')=\gbf(\Gamma)$ and 
$\kbf(\Gamma_1')+\kbf(\Gamma_2')=\kbf(\Gamma)+1$. Since $\kbf(\Gamma_i')\neq 0$ 
in this case, we have $\tbf(\Sigma')=\tbf(\Sigma)$. If the strip does not 
separate $\Sigma$, then there is a connected component of $\Sigma$ which sees 
its genus decreased by $1$ and $\kbf$ increased by $1$, so 
$\tbf(\Sigma')=\tbf(\Sigma)-1$. In any cases we see that 
$\tbf(\Sigma')\le\tbf(\Sigma)$.

Concerning the topology of compact surfaces, we have the following result (see~\cite[p. 101]{GaXu}):
\begin{lem}\label{lem:hom}
Let $\Gamma$ be a compact connected surface with genus $\gbf$ and $\kbf$ boundary connected components. Then, if $\Gamma\in\boO$, $H^1(\Gamma)=\Z^{2\gbf}$; if $\Gamma\in\boO_\partial$, $H^1(\Gamma)=\Z^{2\gbf+\kbf-1}$; and, if $\Gamma\in\boN$, $H^1(\Gamma)=\Z^{\gbf-1}\times\Z_2$ when $\kbf=0$ and 
$H^1(\Gamma)=\Z^{\gbf+\kbf-1}$ when $\kbf\neq 0$.
\end{lem}

Let us start the proof of Theorem~\ref{th:top}. Let $\Gamma=\cup_{i=1}^n\Gamma_i$ be the support of $\lim \vbf(\Sigma_{t_j}^j)$ and $m_i=\tbf(\Gamma_i)$. By Lemma~\ref{lem:hom}, there are $m_i$ curves $\gamma_i^1,\dots,\gamma_i^{m_i}$ in $\Gamma_i$ such that, if $k_1,\dots,k_{m_i}$ are integers such that $k_1\gamma_i^1+\cdots+k_{m_i}\gamma_i^{m_i}$ is homologically trivial in $\Gamma_i$, then $k_l=0$ for each $l$. Notice that we can find $\mu>0$ such that $d(\gamma_i^l,\partial M)>\mu$ for any $i,l$. 

For $\eps>0$, we denote by $V_\eps(X)$ the $\eps$-tubular neighborhood of $X\subset M$:
\[
V_\eps(X)=\{p\in M\mid d(p,X)<\eps\}.
\]
We can choose $\eps_0<\mu$ such that there is a smooth tubular neighborhood retraction of $V_{2\eps_0}(\Gamma)$ onto $\Gamma$.

We have (see \cite[Proposition 2.3]{DeLPe}):
\begin{prop}
Let $\eps\le\eps_0$ be positive. For $j$ sufficiently large and $t$ close enough to $t_j$, we can find a surface $\widetilde\Sigma_t^j$ obtained from $\Sigma_t^j$ through surgery and satisfying:
\begin{itemize}
\item $\widetilde\Sigma_t^j$ is contained in $V_{2\eps} (\Gamma)$;
\item $\widetilde\Sigma_t^j\cap V_\eps(\Gamma)=\Sigma_t^j\cap V_\eps(\Gamma)$.
\end{itemize}
\end{prop}

\begin{proof}
Let $\Ome=V_{2\eps}(\Gamma)\setminus\overline{V_\eps(\Gamma)}$. Since $\Gamma$ is the support of $\lim\vbf(\Sigma_{t_j}^{j})$, we have
\[
\lim_{j\to\infty}\limsup_{t\to t_j}\boH^2(\Sigma_t^j\cap\Ome)=0.
\]
Thus, for $\eta>0$, we can find a positive integer $N$ and $\delta_j>0$ such that
\[
\boH^2(\Sigma_t^j\cap\Ome)<\eta\quad\text{for}\quad j\ge N\quad\text{and}\quad|t-t_j|<\delta_j.
\]
If $\Delta_\sigma=\{p\in M\mid d(p,\Gamma)=\sigma\}$, by the coarea formula, we have
\[
\int_\eps^{2\eps} \boH^1(\Sigma_t^j\cap \Delta_\sigma)d\sigma\le C\boH^2(\Sigma_t^j\cap\Ome)<C\eta,
\]
for some constant $C$ independent of $t$ and $j$. Thus $\boH^1(\Sigma_t^j\cap\Delta_\sigma)\le2C\eta/\eps$ for a set of $\sigma$ of measure at least $\eps/2$. We can then chose $\sigma$ such that $\Delta_\sigma$ intersects $\Sigma_t^j$ transversally and $\boH^1(\Sigma_t^j\cap \Delta_\sigma)\le2C\eta/\eps$.

Notice that there are positive constants $\lambda$ and $K$ such the following is true. For $s\in(0,2\eps)$, any simple curve $\gamma$ in $\Delta_s$ with length less than $\lambda$ satisfies:
\begin{itemize}
\item if $\gamma$ is closed then it bounds an open disk $D\subset\Delta_s$ 
with $\diam(D)\le K\length(\gamma)$; 
\item if the end-points of $\gamma$ are in $\partial\Delta_s$ then there is an 
open half-disk $D\subset\Delta_s$ such that $\partial D\setminus\partial 
M=\gamma\setminus\partial M$ and $\diam(D)\le K\length(\gamma)$.
\end{itemize} 

We may assume that $2C\eta/\eps<\lambda$. Since $\Sigma_t^j\cap\Delta_s$ is transverse, it is the union of a finite number of curves that are either closed or have end-points in $\partial\Delta_s$. In any case, the above properties apply and give the expected result as in~\cite{DeLPe}. Notice that one has to consider first an innermost curve in $\Sigma_t^j\cap\Delta_s$ to do the construction of the proof of \cite[Proposition~2.3]{DeLPe}.
\end{proof}

The proof of Simon's Lifting Lemma~\cite[Proposition~2.1]{DeLPe} allows to lift a closed curve $\gamma\in\Gamma_i$ to a closed curve $\tilde\gamma$ in $\Sigma_{t_j}^{j}$ in a tubular neighborhood of $\gamma$. So the whole proof of \cite[Theorem~1.6]{DeLPe} can be carried out in the $\mu/2$-tubular neighborhood of $\cup_{i,l}\gamma_i^l$. This gives Theorem~\ref{th:top}.

\section{Rigidity of the half-ball}\label{sec:rigid}

In this section, we state and prove a rigidity statement for the Euclidean 
half-ball. More precisely, we have the following result:

\begin{thm}\label{th:rigidhalfball}
Let $M$ be a compact connected $n$-manifold with piecewise smooth boundary 
whose boundary is made up of two smooth pieces, say $\Sigma$ and $\Delta$, such 
that $\Sigma\cap\Delta=\boC(M)$. Consider on $M$ a 
Riemannian metric $g$ such that:
\begin{itemize}
\item $\Ric\ge 0$;
\item the mean curvature $H_\Sigma$ of $\Sigma$ with respect to the 
inward unit normal satisfies $H_\Sigma\ge n-1$;
\item $\Delta$ is totally geodesic and isometric to the Euclidean $(n-1)$-ball 
of radius $R>0$;
\item $\Sigma$ and $\Delta$ meet orthogonally.
\end{itemize}
Then $R\le 1$ and, if $R=1$, $(M,g)$ is isometric to the Euclidean unit half-ball of dimension $n$.
\end{thm}

Notice that \textit{a priori} we do not assume that $M$ is diffeomorphic to the 
Euclidean half-ball.

The rest of this section is devoted to the proof of the above statement. The 
proof is inspired by the work C.~Xia~\cite{Xia} which is based on the work of 
R.C.~Reilly~\cite{Rei}.

On $\Delta$, let $d$ be the distance function to $\partial\Delta$ and consider on $\partial M$ the continuous function
\[
f:p\mapsto
\begin{cases}
(2Rd-d^2)(p)&\text{if }p\in\Delta,\\
0&\text{if }p\in\Sigma.
\end{cases}
\]
Notice that $(2Rd-d^2)$ equals to $R^2$ minus the squared distance 
to the center of $\Delta$. Then let $u$ be the function on $M$ that solves
\[
\begin{cases}\barre\Delta u=-2n&\text{on }M,\\
u=f&\text{on }\partial M.
\end{cases}
\]
We notice that, by elliptic regularity, $u$ is smooth up 
to $\partial M$ except along 
$\partial\Delta=\partial\Sigma=\boC(M)$. Consider the function 
$F=\frac12|\barre\nabla u|^2+2u$ on $M\setminus\boC(M)$. Using Bochner 
formula together with $\tr\barre\nabla^2u=\barre\Delta u=-2n$, we first have
\begin{equation}\label{eq:bochner}
\begin{split}
\barre\Delta F&=(\barre\nabla u,\barre\nabla \barre\Delta u)+|\barre\nabla^2 
u|^2+\Ric(\barre\nabla u,\barre\nabla 
u)+2\barre\Delta u\\
&\ge |\barre\nabla^2u+2g|^2+4n-4n=|\barre\nabla^2u+2g|^2\ge 0\,.
\end{split}
\end{equation}
So $F$ is subharmonic. Let us compute 
$\partial_\nu F$ along $\Sigma$ and $\Delta$. Observe that
\[
\frac12\partial_\nu|\barre\nabla u|^2=\barre\nabla^2u(\nu,\barre\nabla 
u)=\barre\nabla^2u(\nu,\nabla f)+\chi\barre\nabla^2u(\nu,\nu),
\]
where $\chi=\partial_\nu u$. Since $\tr\barre\nabla^2u=-2n$, if $(e_i)_{1\le i\le n-1}$ is an orthonormal basis of $T\partial M$, we have
\begin{align*}
\barre\nabla^2u(\nu,\nu)&=-2n-\sum_i\barre\nabla^2u(e_i,e_i)\\
&=-2n-\sum_i(\barre\nabla_{e_i}(\nabla f+\chi\nu),e_i)\\
&=-2n-\Delta f-\chi H\,,
\end{align*}
where $H=\tr_{\partial M}\ii$. We also have
\[
\barre\nabla^2 u(\nu,\nabla f)=(\barre\nabla_{\nabla f}\barre\nabla 
u,\nu)=-\ii(\nabla f,\nabla f)+(\nabla f,\nabla\chi).
\]	
Therefore
\[
\frac12\partial_\nu|\barre\nabla u|^2=-\ii(\nabla f,\nabla f)+(\nabla f,\nabla\chi)-\chi(2n+\Delta f+\chi H)
\]
and
\[
\partial_\nu F=-\ii(\nabla f,\nabla f)+(\nabla f,\nabla\chi)-\chi(2(n-1)+\Delta 
f+\chi H).
\]
On $\Sigma$, since $f=0$ and $H_\Sigma\ge n-1$, we have 
\begin{equation}\label{eq:parnuFSigma}
\partial_\nu F=-\chi (2(n-1)+\chi H)\le -(n-1)\chi(2+\chi).
\end{equation}
On $\Delta$, since $\Delta$ is totally geodesic and isometric to the 
Euclidean 
$(n-1)$-ball of radius $R$, $\Delta f=-2(n-1)$ and then
\begin{equation}\label{eq:parnuFDelta}
\partial_\nu F=(\nabla f,\nabla\chi).
\end{equation}
Let us remark that at the center $\bar p$ of $\Delta$, we have $u(\bar p
)=f(\bar 
p)=R^2$ and so $F(\bar p)\ge2R^2$. 

\begin{claim}\label{cl:2}
$R\le 1$ and, if $R=1$, $F$ is constant and equal to $2$.
\end{claim}

\begin{proof}[Proof of Claim~\ref{cl:2}]
Let us consider $(q_k)$ a sequence of 
points in $M\setminus\boC(M)$ such that $$\lim 
F(q_k)=\sup_{M\setminus 
\boC(M)} F.$$ We may 
assume that $q_k\to\bar q\,\in M$. Since $F$ is subharmonic, we may even 
assume that $\bar q\in\partial M$. If $\bar q\in \Sigma\cap\Delta$, we will 
prove 
later (see Claim~\ref{cl:corner}) that
\[
\sup_{M\setminus \boC(M)} F=\lim F(q_{k})=2R^2;
\]
thus the maximum value of $F$ is reached also at $\bar p$. So let us assume 
that $\bar q\in\partial M\setminus\boC(M)$. Then necessarily 
$\partial_\nu F(\bar q) \ge 0$. If $\bar q\in \Sigma$, 
\eqref{eq:parnuFSigma} 
implies $-2\le\chi(\bar q)\le 0$ and then $F(\bar 
q)\le2$, since $F=\frac12\chi^2$ on $\Sigma$. Hence 
$2\ge F(\bar q)\ge F(\bar p)\ge 2R^2$; this 
implies $R\le 1$ and, if $R=1$, the 
maximum value of $F$ is reached also at $\bar p$. Thus we assume $\bar 
q\in\Delta$.

On $\Delta$, $F=2R^2+\frac{1}{2}\chi^2$. So, at the maximum, one has 
$\nabla 
\chi(\bar q)=0$ or $\chi(\bar q)=0$. In the first case, \eqref{eq:parnuFDelta} 
implies $\partial_\nu F(\bar q)=0$, so $F$ is constant by the boundary 
maximum principle and at most $2$ by \eqref{eq:parnuFSigma} and 
$F=\frac12\chi^2$ on $\Sigma$. Then $R\le 1$ 
since 
$F(\bar p)\ge 2R^2$. In the case 
$\chi(\bar 
q)=0$, we have $F(\bar q)=2R^2$. As a consequence, $\chi=0$ 
along $\Delta$ and 
thus $\nabla\chi(\bar q)=0$, which, as above, implies that $F$ is 
constant and at most $2$ and $R\le 1$. Moreover, if $R=1$, $F$ 
is 
constant and equal to~$2$.
\end{proof}

Let us now finish the proof of Theorem~\ref{th:rigidhalfball} in the 
case $R=1$ and $F$ equals $2$. Because of~\eqref{eq:bochner}, 
\eqref{eq:parnuFSigma}, and \eqref{eq:parnuFDelta}, 
$\barre\nabla^2u=-2g$ and $\Ric(\barre\nabla u,\barre\nabla u)=0$ 
on $M$, $\chi=0$ on $\Delta$, and $\chi=-2$ and $H=n-1$ on 
$\Sigma$. Moreover, since $F\equiv 2$, $|\barre\nabla 
u|^2=4(1-u)$, thus $\barre\nabla u$ has constant norm along a level set 
of $u$. So the second fundamental form of such a level set is given by 
$-(\barre\nabla_X\frac{\barre\nabla 
u}{2\sqrt{1-u}},Y)=\frac{-1}{2\sqrt{1-u}}\barre\nabla^2u(X,Y)=\frac{1}{\sqrt{1-u}}(X,Y)$
 for $X,Y$ tangent to the level set. By the maximum principle, $u\ge 0$ 
and $\{u=0\}=\Sigma$. 

At the center $\bar p$ of $\Delta$, we have $u(\bar p)=1$ and $\barre\nabla 
u(\bar p)=0$. If $\gamma$ is a unit geodesic starting at $\bar p$, we have 
$\frac{d}{ds}u(\gamma(s))_{|s=0}=0$ and $\frac{d^2}{ds^2}u(\gamma(s))=-2$ so 
that $u(\gamma(s))=1-s^2$. Thus we have proved that~$M$ is the image of 
the unit half-ball in $T_{\bar p}M$ by the exponential map. Let us fix 
$v$, $y_1$ and $y_2$ in $T_{\bar p}M$ with $|v|=1$, $v$
pointing inward and $y_i$ normal to $v$. Let us define $Y_i(s)=d 
\exp_{\bar p}(sv)(sy_i)$ two Jacobi fields along $s\mapsto\exp_{\bar 
p}(sv)$ and $U=d\exp_{\bar p}(sv)v$ the unit speed vector 
along the geodesic (notice that $Y_1$ and $Y_2$ are tangent to the level set 
$\{u=1-s^2\}$). Then, by the above computation of the second fundamental 
forms of the level sets, we have
\begin{align*}
\frac{d}{ds}(Y_1,Y_2)=(\barre\nabla_UY_1,Y_2)+(Y_1,\barre\nabla_UY_2)= 
(\barre\nabla_{Y_1}U,Y_2)+(Y_1,\barre\nabla_{Y_2}U)=\frac2s(Y_1,Y_2).
\end{align*}
Thus $(Y_1,Y_2)=s^2(y_1,y_2)$ and $\exp_{\bar p}$ is an isometry from the unit half-ball in $T_{\bar p}M$ to $M$. This finishes the proof of Theorem~\ref{th:rigidhalfball} when the following Claim is proved.

\begin{claim}\label{cl:corner}
If $\bar q\in\Sigma\cap\Delta=\boC(M)$, then $\lim F(q_k)=2$.
\end{claim}

\begin{proof}[Proof of Claim~\ref{cl:corner}]
First, let $\delta$ denote the distance function to $\Sigma$. For 
$\eps>0$ small enough, $\delta$ is smooth on $\{\delta\le\eps\}$ and 
equal to $d$ on $\Delta\cap\{\delta\le \eps\}$. Let $(e_i)_{1\le i\le n-1}$ be 
an orthonormal basis of~$\Sigma$. Completing it by $\barre\nabla\delta$, 
we can compute
\[
\barre\Delta\delta=\sum_i(\barre\nabla_{e_i}\barre\nabla\delta,e_i)= 
-\sum_i\ii_\Sigma(e_i,e_i)\le-(n-1)\quad\text{on}\quad\Sigma.
\]
Reducing $\eps>0$ if necessary, we may assume that 
$\barre\Delta\delta\le-(n-\frac32)$ on $\{\delta\le\eps\}$. Since $u$ is 
bounded in $M$, we can find a constant $c>0$ such that $u\le c\delta$ 
on $\p\{\delta\le\eps\}$ and $2n\le c(n-\frac32)$. We then have $\barre\Delta 
u=-2n\ge-c(n-\frac32)\ge\barre\Delta(c\delta)$ on 
$\{\delta\le\eps\}$. Hence, by the maximum principle, we have $0\le u\le 
c\delta$ on $\{\delta\le\eps\}$.

Let us consider a local chart of $M$ near $\bar q$: we see a neighborhood of 
$\bar q$ in $M$ as a neighborhood of the origin in $\R_+^2\times\R^{n-2}$ such 
that $\{x_2=0\}$ corresponds to $\Delta$ and $\{x_1=0\}$ corresponds 
to~$\Sigma$. The Riemannian metric is then given by $g=(g_{ij})$ and, since 
$\Sigma$ and $\Delta$ meet orthogonally, we may assume that $g$ is the identity 
at the origin. The sequence of points $(q_k)$ can then be written as 
$q_k=a_k+r_k\theta_k$, where $a_k$ is the projection of $q_k$ on 
$\{(0,0)\}\times\R^{n-2}$, $r_k\in\R_+^*$ and $\theta_k\in\S^1\times\{0\}$. We 
have $a_k\to 0$, $r_k\to 0$, and we can assume that 
$\theta_k\to\bar\theta\in\S^1\times\{0\}$. We can see $u$ as being defined on 
$\R_+^2\times\R^{n-2}$ close to the origin and we define 
$u_k(x)=\frac1{r_k}u(r_k x+a_k)$. The function $u$ solves
\[
\frac1{\sqrt{\det g}}\sum_{i,j}\partial_i(\sqrt{\det 
g}\,g^{ij}\partial_ju)=-2n\,,
\]
where $g^{ij}$ is the inverse matrix of $g$. As a consequence, $u_k$ solves
\[
\frac1{\sqrt{\det g_k}}\sum_{i,j}\partial_i(\sqrt{\det 
g_k}\,g_k^{ij}\partial_ju_k)=-2nr_k\,,
\]
where $g_k(x)=g(r_kx+a_k)$. We see that $(g_k)$ goes uniformly to the identity 
on the quarter-ball of radius $\rho$ for any $\rho>0$. Moreover, we 
know that $0\le u\le c\delta$ and $\delta\le Kx_1$ for some $K>0$; thus $0\le 
u_k\le cK x_1$. This implies that the $u_k$'s are uniformly bounded in the 
quarter-ball of radius $\rho$.

On the boundary, $u_k=0$ on $\{x_1=0\}$ and 
$u_k(x)=\frac1{r_k}(2Rd-d^2)(r_k 
x+a_k)$ on $\{x_2=0\}$; we notice that the last boundary value 
converges uniformly to $2Rx_1$. Let $D_{\rho}$ be the quarter-disk in 
$\R_+^2$ of radius $\rho$. Because of the above properties, we can 
apply Schauder estimates up to the boundary \cite[Corollary~6.7]{GiTr} to prove 
that $u_k$ has uniformly bounded $C^{2,\alpha}$-norm in 
$(D_{\rho}\setminus D_\eps)\times[-\rho,\rho]^{n-2}$ 
for any $0<\eps<\rho$. As a consequence, we can consider a subsequence 
which converges to a function $\bar u$ defined on 
$(\R_+^2\setminus\{(0,0)\})\times \R^{n-2}$ (the convergence is $C^2$-uniform 
on any compact subsets). Moreover, $\bar u$ solves $\Delta_0\bar u=0$ 
(the 
Euclidean Laplacian) and is equal to $0$ on $\{x_1=0\}$ and to $2Rx_1$ 
on 
$\{x_2=0\}$. We also notice that $0\le\bar u\le cKx_1$, so the function 
extends continuously by $0$ to $\{(0,0)\}\times\R^{n-2}$. We assert that 
actually $\bar u(x)=2Rx_1$. Let $\bar v=\bar u-2Rx_1$; it is a 
solution 
to $\Delta_0\bar v=0$ on $\R_+^2\times\R^{n-2}$ which extends 
continuously 
by $0$ on the boundary. Thus, by Schwarz reflection along the boundary, 
we can extend it to a solution to $\Delta_0\bar v=0$ on $\R^n$. We also 
notice that $|\bar v|\le \max(cK,2R)|x_1|$. This implies that 
$|\nabla \bar v|$ is 
uniformly bounded on $\R^n$ and Liouville theorem gives that there is a 
vector $a\in\R^n$ and a constant $b$ such that $\bar v(x)=(a,x)+b$. Since $\bar 
v=0$ on $\{x_1=0\}$ and $\{x_2=0\}$, we have $b=0$, $a=0$, and 
$\bar u=2Rx_1$.

Finally, since $u_k\to\bar u$ $C^2$-uniformly on $(D_2\setminus D_{1/2})\times 
[-1,1]^{n-2}$, we obtain
\[
F(q_k)=\frac12|\barre\nabla 
u|^2(q_k)+2u(q_k)=\frac12|\nabla_{g_k} 
u_k|^2(\theta_k)+2u(q_k)\longrightarrow\frac12|\nabla_{0}\bar 
u|^2(\bar\theta)+2u(\bar q)=2R^2.
\]
This finishes the proof of Claim~\ref{cl:corner} and the proof of Theorem~\ref{th:rigidhalfball}.
\end{proof}

\section{Least area minimal disk}\label{sec:least}

In this section, we prove Theorem~\ref{th:main}. So let us consider $(B,g)$ to 
be the $3$-ball with a Riemannian metric with non-negative Ricci curvature and 
$\ii_{\partial B}\ge g_{|\partial B}$. It was brought to our attention 
by R. Haslhofer that part of the below argument also appears in his recent work 
with D. Ketover~\cite{HaKe} (see Proposition~2.4 therein).

The first step consists in proving the existence of a free boundary minimal 
disk in $B$. This is a consequence of the work of Jost \cite[Theorem 
4.1]{Jos2}. Let us briefly recall the construction. The family $\{\Gamma_t\}$ 
defined by
\[
\Gamma_t=B\cap\{z=t\},\text{ for }t\in[-1,1],
\]
is a sweepout of the ball $B$. Let $\Lambda_0$ be the smallest saturated set of 
sweepouts containing $\{\Gamma_t\}$. Actually, for any 
$\{\Sigma_t\}\in\Lambda_0$, there is $t_0\in [-1,1]$ such that $\Sigma_{t_0}$ 
separates $B$ into two parts of equal volume so that $|\Sigma_{t_0}|$ is 
larger than the isoperimetric profile of $(B,g)$ for half the volume 
of~$(B,g)$. As a consequence, $W(\Lambda_0,g)>0$ and 
Theorem~\ref{th:min_max} implies the existence of free boundary minimal 
surfaces in $(B,g)$. Since $\tbf(\Gamma_t)=0$ for any $t$, Theorem~\ref{th:top} 
ensures that these minimal surfaces are either disks, spheres or projective 
planes. Besides, these minimal surfaces must have a nonempty 
boundary by Fraser and Li~\cite[Lemma 2.2]{FrLi} (notice that projective planes 
are forbidden just by the topology of the ambient space). Thus we have the 
existence of a free boundary minimal disk in $(B,g)$.

By Fraser and Li~\cite[Theorem~1.2]{FrLi}, the space of such minimal disks is compact. So there is a free boundary minimal disk $\Delta$ such that
\[
|\Delta|=\inf\{|D|;D\text{ is a free boundary minimal disk in }(B,g)\}.
\]
Since a free boundary minimal disk can be produced by min-max, we have 
$|\Delta|\le W(\Lambda_0,g)$. Because of the ambient geometry and using 
$u\equiv 1$ in \eqref{eq:stab}, $\Delta$ cannot be stable.

\begin{prop}\label{prop:sweepout}
There is a sweepout $\{\Sigma_t\}_{t\in[-1,1]}\in \Lambda_0$ with the following 
properties:
\begin{itemize}
\item $\mathcal{F}(\{\Sigma_t\})=|\Delta|$;
\item $\{\Sigma_t\}_{t\in(-\eps,\eps)}$ is a piece of the nice foliation 
associated with $\Delta$, for some $\eps>0$, such that $\Sigma_0=\Delta$;
\item for any $t\neq 0$, $\boH^2(\Sigma_t)<|\Delta|$.
\end{itemize}
\end{prop}

\begin{proof}
Let us consider $\{S_t\}_{t\in(-2\eps,2\eps)}$ be the nice foliation associated 
with $\Delta$ (Lemma~\ref{lem:nicefol}). We notice that $|S_t|<|\Delta|$ for 
$t\neq 0$. We fix $\eta>0$ such that $\max(|S_\eps|,|S_{-\eps}|)\le 
|\Delta|-\eta$. Let $\boN=\cup_{t\in(-\eps,\eps)}S_t$ be a neighborhood of 
$\Delta$ and let $M_+$ and $M_-$ be the connected components of $B\setminus 
\boN$ such that $S_{\pm\eps}\subset \partial M_\pm$. Let us consider $M_+$ 
(similar arguments can be applied to $M_-$). $M_+$ is diffeomorphic to $B^+$. 
We are going to study the constrained min-max construction associated 
with $\partial_0M_+=S_{\eps}$, $\partial_+M_+=M_+\cap\partial B$ and the 
saturated set $\Lambda_0^+$ (see Example~\ref{ex:lambda+}). On 
$\cup_{t\in[\eps,2\eps)}S_t$, we consider the function $f$ defined by 
$f(p)=t-\eps$ 
if $p\in S_t$. The function $f$ satisfies Definition~\ref{def:freemcf}, so 
$M_+$ has a free boundary mean-convex foliation near $\partial_0M_+$.
Moreover, $\ii_{\partial_+M_+}\ge 
g_{|\partial_+M_+}$. Therefore, if $W(\Lambda_0^+,g)>|S_{\eps}|$, there 
is at least one free boundary minimal surface $S$ in $M_+$ with free 
boundary in $\partial_+M_+$. As a consequence, $\Delta$ and $S$ are two 
disjoint minimal surfaces in $(B,g)$, which is not possible by the 
Frankel property \cite[Lemma~2.4]{FrLi}.

Thus there is a constrained sweepout $\{\Sigma_t^+\}_{t\in[0,1]}\in\Lambda_0^+$ 
of $M_+$ with $\Sigma_0^+=S_{\eps}$ such that 
\[
\boF(\{\Sigma_t^+\})\le|S_\eps|+\eta/2<|\Delta|.
\]
Similarly, there is a constrained sweepout $\{\Sigma_t^-\}_{t\in[0,1]}$ of 
$M_-$ with $\Sigma_0^-=S_{-\eps}$ such that 
\[
\boF(\{\Sigma_t^-\})<|\Delta|.
\]
We can then define
\[
\Sigma_t=\begin{cases}
\Sigma_{\theta_-(t)}^-&\text{if }t\in[-1,-\eps],\\
S_t&\text{if }t\in[-\eps,\eps],\\
\Sigma_{\theta_+(t)}^+&\text{if }t\in[\eps,1],\\
\end{cases}
\]
where $\theta^-(t)=-\frac{t+\eps}{1-\eps}$ and 
$\theta^+(t)=\frac{t-\eps}{1-\eps}$. $\{\Sigma_t\}_{t\in[-1,1]}$ is then a 
sweepout in $\Lambda_0$ with the expected properties.
\end{proof}

A consequence of the above result is that $|\Delta|=W(\Lambda_0,g)$.

\begin{prop}\label{prop:index}
$\Delta$ has index one.
\end{prop}

\begin{proof}
We already know that $\Delta$ has index at least one. Let us assume that it has 
index at least two. Let $\lambda_1$ and $\lambda_2$ be the first two negative 
eigenvalues of the Jacobi operator and $\phi_1$ and~$\phi_2$ be the associated 
eigenfunctions. Consider $\{\Sigma_t\}_{t\in[-1,1]}$ the sweepout given 
by the preceding proposition. Using the notations of Lemma~\ref{lem:nicefol}, 
for a function $u$ defined on $\Delta$, we consider $X_u(\Delta)$ its 
graph in a neighborhood of $\Delta$. We know that the nice foliation, 
and thus the above sweepout, is given by 
\[
\Sigma_t=X_{u_t}(\Delta),\text{ for }t\in[-\eps,\eps],
\]
where $\partial_t{u_t}_{|t=0}=\phi_1$. Let us define 
\[
\Sigma_{t,s}=X_{u_t+s\phi_2}(\Delta), \text{ for }(t,s)\in[-\eps,\eps]^2.
\]
Since $\phi_1$ and $\phi_2$ are associated to negative eigenvalues, the Hessian 
of $F:(t,s)\mapsto\boH^2(\Sigma_{t,s})$ at $(0,0)$ is negative definite (notice 
that $(0,0)$ is a critical point of $F$ and see~\eqref{eq:stab}). Since 
$F(t,0)<F(0,0)$ for $t\neq 0$, there is a function $\tilde 
s:[-\eps,\eps]\to[-\eps,\eps]$ such that $\tilde s(\pm\eps)=0$ and $F(t,\tilde 
s(t))<F(0,0)=|\Delta|$ for any $t\in[-\eps,\eps]$. As a consequence, the 
sweepout $\{\widetilde\Sigma_t\}_{t\in[-1,1]}$ defined by
\[
\widetilde\Sigma_t=
\begin{cases}
\Sigma_t&\text{if }t\notin[-\eps,\eps]\\
\Sigma_{t,\tilde s(t)}&\text{if }t\in[-\eps,\eps]
\end{cases}
\]
satisfies to $\boF(\{\widetilde\Sigma_t\})<|\Delta|=W(\Lambda_0,g)$, which is 
impossible since $\{\widetilde\Sigma_t\}\in \Lambda_0$. So $\Delta$ has index 
one.
\end{proof}
Let us remark that a general index estimate can be found in \cite{Frz}.
The first two items of Theorem~\ref{th:main} are then proven. The third one is 
then a consequence of Theorem~\ref{th:Men} by the second author 
\cite{Men}.

When $L(\partial\Delta)=2\pi$, the equality case of Theorem~\ref{th:Men}
gives that $\Delta$ is totally geodesic and isometric to the Euclidean unit 
disk. So Theorem~\ref{th:rigidhalfball} gives that both sides of $\Delta$ are 
isometric to the unit half-ball, and then $(B,g)$ is isometric to the unit 
ball, in Euclidean space. This finishes the proof of~Theorem~\ref{th:main}.

\begin{rmk}In case the sectional curvature of $g$ is non-negative, we 
can use Proposition~\ref{prop:area} and are able to estimate the area of 
$\Delta$ by $|\Delta|\le \frac12L(\partial\Delta)\le \pi$. In case 
$|\Delta|=\pi$, the rigidity is 
then a direct consequence of the above arguments.\end{rmk}

\appendix

\section{Construction of a free boundary mean-convex foliation}

\label{sec:foliat}

In Section~\ref{subsec:constrained}, we define the notion of a free 
boundary mean-convex foliation near $\partial_0M$. In this section, we 
construct such a foliation under some simple hypotheses. If $\nu_0$ and $\nu_+$ 
are the inward unit conormal for $\partial_0M$ and $\partial_+M$ along 
$\boC(M)$, the \emph{inner angle} between $\partial_0 M$ and $\partial_+M$ is 
the one between $\nu_0$ and $\nu_+$.

\begin{prop}\label{prop:app}
Let $(M,g)$ be a Riemannian $n$-manifold with piecewise smooth boundary whose 
boundary is the 
union $\partial M=\partial_0M\cup\partial_+M$ of two smooth parts. If 
$\partial_0 M$ has positive mean curvature and the inner angle between 
$\partial_0 M$ and $\partial_+M$ is at most $\pi/2$, then $M$ has a free 
boundary mean-convex foliation near $\partial_0M$.
\end{prop}

The hypotheses in the above result should be compared with \cite[Definition 
2.2]{GuWaZh} which is used to do a constrained min-max construction 
in the Almgren-Pitts setting.

The rest of the appendix is devoted to the proof of Proposition~\ref{prop:app}. 
It is made of several steps. In 
order to do the construction, we start with a parametrization of a 
neighborhood of 
$\partial_0M$ in $M$ by $\partial_0M\times[0,1)$ with a second coordinate $y$ 
such that $\partial_y$ is unitary and orthogonal to $\boC(M)$ along 
$\boC(M)$. 

\subsection{Improving the coordinate system}
In order to deal with points where the inner angle is $\pi/2$, we need to 
choose a
better coordinate $y$. So let $u$ be a positive function on 
$\partial_0M$ such that $u=1$ on $\boC(M)$. Notice that we can choose the value 
of $\partial_\nu u$ as we want. Let us define the map
$F:\partial_0M\times [0,\eps)\to \partial_0M\times [0,1); (a,z)\mapsto 
(a,u(a)z)$. Since $u$ is positive, $F$ is a diffeomorphism on its image and 
this 
defines a new parametrization of a neighborhood of 
$\partial_0M$ in $M$. 

We can consider a parametrization of a neighborhood of $\boC( M)$ in 
$\partial_0M$ by $\boC(M)\times[0,\eps)$ such that the second coordinate 
function $x$ is the distance function in $\partial_0 M$ to $\boC(M)$. 

With the $y$ coordinate, this gives a parametrization of a neighborhood of 
$\boC(M)$ in $M$ and defines two vector fields: $\partial_y$ and $\partial_x'$. 
With 
the $z$ coordinate, this gives a second parametrization of this neighborhood 
and 
two new vector fields : $\partial_z$ and $\partial_x$. We have the equalities
$\partial_z=u(a)\partial_y$ and $\partial_x=\partial_x'+ 
z\partial_xu(a)\partial_y$ (notice that $\partial_x=\partial_x'$ on 
$\partial_0M$).

Since $u=1$ on $\boC(M)$, $\partial_z=\partial_y$ and then 
$g(\nabla_{\partial_x}\partial_x,\partial_z)= 
g(\nabla_{\partial_x'}\partial_x',\partial_y)$ on $\boC(M)$. We also have 
$g(\partial_x,\partial_z)=(\partial_x'+z\partial_xu(a)\partial_y,\partial_y)$ 
on $\partial_+M$. So, on $\boC(M)$,
\[
\partial_zg(\partial_x,\partial_z)=\partial_yg(\partial_x',\partial_y) 
+\partial_xu(a)=\partial_yg(\partial_x',\partial_y) -\partial_\nu u(a)\,.
\]
So, adjusting the value of $\partial_\nu u$, we can assume, for any $K\in 
\R_+$, that, on $\boC(M)$, 
\begin{equation}\label{eq:coord_improve}
\partial_zg(\partial_x,\partial_z)\ge K 
|g(\nabla_{\partial_x}\partial_x,\partial_z)|\,.
\end{equation}

Let us also notice that since $\partial_0M$ has positive mean curvature, 
$\{z=t\}$ has positive mean curvature (in the $\partial_z$ direction) for small 
$t$. We are now using the parametrization of a neighborhood of $\boC(M)$ in $M$ 
by the coordinates $(q,x,z)\in\boC(M)\times[0,\eps)\times[0,\eps)$ (see 
Figure~\ref{fig:fig1}). 
We then have a positive function $H :\boC(M)\times[0,\eps)^2\to \R$ such 
that $H(q,x,t)$ is the mean curvature of $\{z=t\}$ at the point parametrized by 
$(q,x,t)$ with respect to the $\partial_z$ direction. The parametrization of 
this neighborhood gives a decomposition at $(q,x,z)$ of any vector field $Z$ as 
$Z=Z^x\partial_x+Z^z\partial_z+Z^\top$ with $Z^\top$ is tangent to 
$\boC(M)\times\{(x,z)\}$.

\subsection{Construction of the foliation}
By hypotheses, 
along $\boC(M)$ the vector fields $\partial_x$ and $\partial_z$ are 
unitary, normal to $\boC(M)$ and 
$(\partial_x,\partial_z)=\cos\alpha$ where $\alpha=\alpha(q)\in(0,\pi/2]$ is 
the inner angle at $q\in \boC(M)$.

For $\eta>0$, let us define $N_\eta=\{x^2+z^2\le \eta^2\}$ and 
$E=\{z\ge x\}$ 
in $\boC(M)\times[0,\eps)\times[0,\eps)$. Let us also denote by $\delta$ the 
distance function to $\partial_+M$, \textit{i.e.} $\{x=0\}$ in the 
parametrization. If 
$\eta$ is small enough, $\delta$ is smooth in $E\cap N_\eta$ 
(see~\ref{sec:mettodist} below).

If $q\in \boC(M)$, by hypotheses, the Riemannian metric at $q$ can be 
written 
as $ds^2+dx^2+2\cos\alpha dxdz+dz^2$ where $ds^2$ is some scalar product on 
$T_q\boC(M)$ and $\alpha\in(0,\pi/2]$. As a consequence, as $(p,x,z)\in E$ 
converges to $(q,0,0)$, 
$\nabla\delta$ converges to $X_\alpha=\frac1{\sin\alpha}(\partial_x-\cos\alpha 
\partial_z)$ (see \ref{sec:mettodist} below).

Let $\psi$ be a smooth function defined on $\R_+$ such that $\psi(0)=0$, 
$\psi=1$ on $[1,+\infty)$ and $\psi'>0$ on $[0,1)$ (the expression of $\psi$ 
will be precised later). We then consider the vector field $X$ on $N_\eta$ 
defined by
\[
X=\psi(\frac xz)\partial_x+(1-\psi(\frac xz))\nabla\delta
\]
Notice that $X=\partial_x$ outside $E$ so $X$ is a smooth vector field in 
$N_\eta\setminus \boC(M)$. On $E\cap N_\eta$, $X=(\psi(\frac 
xz)+(1-\psi(\frac xz))(\nabla\delta)^x)\partial_x +(1-\psi(\frac 
xz))(\nabla\delta)^z\partial_z+(1-\psi(\frac 
xz))(\nabla\delta)^\top$. Since $\nabla\delta$ is close to 
$X_\alpha$ for some $\alpha\in(0,\pi/2]$, there is $M>0$ and, for any 
$\mu>0$, 
there is an $\eta>0$ such that in $N_\eta$
\begin{gather}\label{eq:app1}
1-\mu \le X^x \le M\\
-M\le X^z\le \eps\label{eq:app2}\\
|X^\top|\le \mu\label{eq:app3}
\end{gather}
As a consequence for $\eta$ small the coordinate $X^x$ never vanishes and we 
can define the vector field
\[
Y=-\frac X{X^x}
\]
whose $Y^x$ coordinate is $-1$.

Moreover, under the above estimates, there is $\eta_1$ such that any integral 
curve of $Y$ starting from a point in 
$E\cap N_{\eta_1}$ stays in $E\cap N_\eta$ for positive time up to a time when 
it hits $\partial_+M=\{x=0\}$. Notice that, when $x=0$, $Y$ is a positive 
multiple of $-\nabla\delta$. So 
such an integral curve hits $\partial_+M$ orthogonally. Besides there is 
$\eta_2<\eta_1$ such that, in negative time, any integral curve of $Y$ starting 
from $E\cap N_{\eta_2}$ hits $\{x=z\}$ in $N_{\eta_1}$.

Let us denote by $(\phi_t)$ the flow associated to the vector field $Y$ (when 
it 
is defined). For example, if $p\in E\cap N_{\eta_1}$ has coordinates 
$(q,x,z)$, $\phi_t(p)$ is defined on the interval of positive time 
$[0,x]$ since the derivative of the $x$ coordinates along $Y$ is $Y^x=-1$ and 
$\phi_t(p)$ hits $\{x=0\}$ at time $x$. For $s\in 
(0,\eta_1/\sqrt 2)$ we define 
\[
S_s=(\{z=s\}\setminus E)\bigcup \left(\cup_{t\in [0,s]}\phi_t(\{x=s=z\})\right)
\]
Notice that the two pieces of $S_s$ are glued along $\{x=s=z\}$ and this gluing 
is smooth since $Y$ extend smoothly to $-\partial_x$ outside $E$. So $S_s$ is a 
smooth hypersurface which is normal to $\partial_+M$ since it contains integral 
curve of the vectorfield $Y$. By construction the function $f$ defined near 
$\partial_0M$ by 
\[
p\mapsto\begin{cases}
0&\text{if }p\in\partial_0M\\
s&\text{if }p\in S_s
\end{cases}
\]
is well defined, continuous in a neighborhood of $\partial_0M$, and is a 
smooth submersion outside of $\boC(M)$. Moreover $\partial_\nu f=0$ along 
$\partial_+M$, since $S_s$ contains integral curves of $Y$ that hits 
$\partial_+M$ orthogonally.
The only aspect that we have to check is that $S_s$ has positive mean curvature 
for $s$ small. 

\subsection{Study of the mean curvature} If $S_s$ does not have positive 
mean curvature for small $s$, there are two sequences $s_n\to 0$ and 
$p_n\in S_{s_n}$ such that the mean curvature of $S_{s_n}$ at $p_n$ is 
non-positive. By construction, the mean curvature of $\{z=s_n\}$ is positive 
for $n$ large, so $p_n\in E$ and we may assume that $p_n\to \bar p\in \boC(M)$. 
In order to analyze the mean curvature of $S_{s_n}$ at $p_n$ we fix a chart of 
$\boC(M)$ near $\bar p$: so we identify $\boC(M)$ with a small ball in 
$\R^{n-2}$ such that $\bar p$ is the origin and the Euclidean basis 
$(\partial_i)_{1\le i\le n-2}$ is orthonormal for the induced metric at the 
origin.

Let us denote by $B_r$ the open ball of radius $r$ in $\R^{n-2}$ 
centered at the origin. Let $A$ be larger than $\cot\alpha$ for any 
inner angle $\alpha$ on $\boC(M)$.
We consider $V=\{(q,x,z)\in B_2\times 
\R_+^2\mid \frac12 x< z< 1-A(x-2)\text{ and }z> \frac12\}$ and 
$U=\{(q,x,z)\in 
B_4\times \R_+^2\mid \frac14 x< z< 15+10A-x)\}$ (see Figure~\ref{fig:fig2}). 

\begin{figure}
\centering
\resizebox{0.6\linewidth}{!}{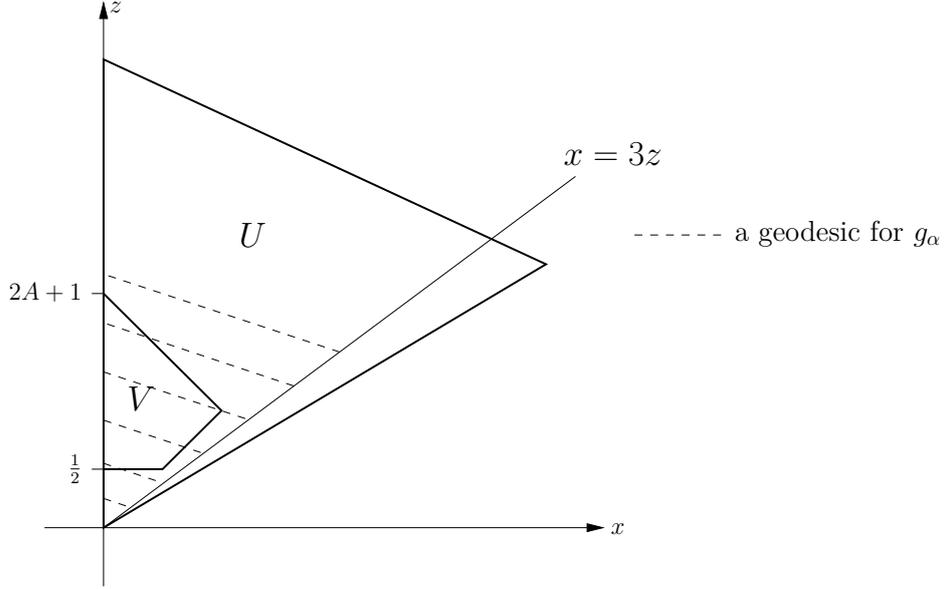}
\caption{The geodesics of $g_\alpha$}\label{fig:fig2}
\end{figure}
We then consider 
$\boG^k$ the set of 
$C^k$-Riemannian 
metric on $U$. Let $\boZ^k$ be the set of $C^k$-vector fields $Z$ in $V$ 
of the form 
$Z=Z^\top+Z^x\partial_x+Z^z\partial_z$ with $Z^\top\in \R^{n-2}\times\{(0,0)\}$ 
such that $Z^x> \frac12$ and $-\frac A 2< Z^z< \frac14$ and $|Z^\top|< 
\frac12$.
For $Z\in \boZ^k$, we define on $V$ the vector field 
\[
\boY(Z)=-\frac{\psi(\frac xz)\partial_x+(1-\psi(\frac xz))Z}{\psi(\frac 
xz)+(1-\psi(\frac xz))Z^x}\,.
\]
By the assumption on $Z^x$, the above vectorfield $\boY(Z)$ is well defined on 
$V$ and has coordinate $-1$ in front of $\partial_x$. As a consequence of the 
bound on $Z$, a solution of $\dot p=\boY(Z)(p)$ starting from $(q,1,1)\in 
B_1\times \R_+^2$ exists and lives in $V$ up to time $t=1$ where it exits $V$ 
by $V_0=V\cap\{x=0\}$. We want to be more precise about the flow associated to 
$\boY(Z)$.

\subsection{The vector field to flow operator} We define 
$\boX^k=\{X\in C^k(V,\R^n)\mid X^x=-1,-\frac12< X^z< A,|X^\top|<1\}$, 
$\boT=\{F\in C^0([0,1],\R^n)\mid F(0)\in B_1\times\{(1,1)\}\text{ and } 
\forall 
t\in[0,1], 
F^x(t)=1-t, 
F(t)\in V\}$ and $T\boT=\{G\in C^0([0,1],\R^n)\mid G^z(0)=0 \text{ and 
}\forall 
t\in[0,1], 
G^x(t)=0\}$. We also define the map
\[
B:\begin{array}{ccc}
B_1\times \boX^k\times \boT&\to& T\boT\\
(q,X,F)&\mapsto & \left(t\mapsto F(t)-(q,1,1)-\int_0^t 
X(F(s))ds\right)
\end{array}
\]
Notice that because of the assumption on $X\in\boX^k$, $B(q,X,F)$ belongs to 
$T\boT$ for $(q,X,F)\in B_1\times \boX^k\times \boT$.
Moreover, for $(q,X)\in B_1\times \boX^k$, the solution of the ode $\dot 
p=X(p)$ starting at
$(q,1,1)$ is defined up to time $1$ and given by $F$ satisfying $B(q,X,F)=0$. 
Reciprocally if $B(\bar q, \barre X,\barre F)=0$, 
$\barre F$ is the solution of $\dot 
p=\barre X(p)$ starting at
$(\bar q,1,1)$. As a consequence $\barre F$ is $C^{k+1}$. The map $B$ is $C^k$ 
and its differential with 
respect to $F$ at $(\bar q, \barre X,\barre F)$ is given by
\[
D_FB(\bar q, \barre X,\barre F):\begin{array}{ccc}
T\boT&\to& T\boT\\
G&\mapsto& \left(t\mapsto G(t)-\int_0^t 
D \barre X(\barre F(s))(G(s))ds\right)
\end{array}
\]
The invertibility of this differential is given by the study of 
the linear ode 
$\dot p=D \barre X(\barre F(t))(p)$. Since $t\mapsto D \barre X(\barre 
F(t))$ is continuous and has vanishing $x$ coordinate, $D_FB( \bar q, \barre 
X,\barre F)$ is invertible. As a 
consequence for any $(\bar q,\barre X,\barre F)$ such that $B(\bar q,\barre 
X,\barre F)=0$ there is a $C^k$ map $\Phi$ defined near $(\bar q,\bar X)$ such 
that $B(q,X,\Phi(q,X))=0$. As a consequence $\Phi(q,X)$ is $C^{k+1}$ in $t$. 
Moreover, since $\partial_t \Phi(q,X)(t)=X(\Phi(q,X)(t))$, its time derivatives 
depends in a $C^k$ way of $(q,X)$. So the map : 
$(q,X,t)\mapsto\Phi(q,X)(t)$ is $C^k$.

\subsection{The mean curvature map} So for $Z\in \boZ^k$, $\boY(Z)\in \boX^k$ 
and the map $\Psi_Z 
:(p,t)\in B_1\times[0,1]\mapsto 
\Phi(q,\boY(Z))(t)\in V$ is $C^k$ and solves
\begin{equation}\label{eq:sys}
\begin{cases}
\partial_t\Psi_Z(q,t)=\boY(Z)(\Psi_Z(q,t))\\
\Psi_Z(q,0)=(q,1,1)
\end{cases}
\end{equation}
The map $\Psi_Z$ is a smooth embedding and we can then define the mean 
curvature 
map $\boH : \boG^k\times \boZ^k\times [0,1]\to \R$ such that 
$\boH(h,Z,t)$ is the mean curvature of $\Psi_Z(B_1\times[0,1])$ at 
$\Psi_Z(0,t)$ 
with respect to the metric $h$ and the unit normal that points in the upward 
direction near $\Psi_Z(0,0)$ (see Figure~\ref{fig:fig1}). Let us notice that, 
by construction, the map 
$\boH$ is $C^2$ when $k\ge 4$.

For $r\in \R^{n-2}$ and $s>0$, we can define the map $J_{r,s}:U\to 
\R^{n-2}\times \R_+^2; (q,x,z)\mapsto (r+sq,sx,sz)$. We notice that when $r$ 
and $s$ are small $J_{r,s}$ maps $U$ into the neighborhood of $\bar p=0$. We 
can 
then define $g_{r,s}=s^{-2}J_{r,s}^*g$. The maps $(r,s)\mapsto g_{r,s}$ extends 
smoothly when $s=0$ by 
a constant metric $g_{r,0}=g(r,0,0)$. So for all $s$ and $r$ small, $g_{r,s}$ 
is close to $g_{0,0}= g_{\bar \alpha}=\delta^2+dx^2+2\cos\bar\alpha dxdz+dz^2$ 
where 
$\delta^2$ 
is the Euclidean metric on $\R^{n-2}$ and $\bar \alpha=\alpha(\bar p)$.

\subsection{The metric to gradient of the distance 
operator}\label{sec:mettodist}
In this subsection, we study the operator that associates to $h\in 
\boG^{2k+1}$, the vector field $\nabla^h\delta_h$ where $\delta_h$ is the 
distance to $\{x=0\}$ with respect to $h$.

 In order to 
study this operator, we first solve the geodesic flow for a metric $h\in 
\boG^{2k+1}$ 
close to $g_{\bar \alpha}$ starting orthogonally from a point in 
$B_3\times\{0\}\times(\frac14,2+2A)$. For 
$h\in \boG^{2k+1}$, the covariant derivative can be written 
$\nabla^h_XY=\nabla^0_XY+\Gamma^h(X,Y)$ where $\nabla^0$ is the usual covariant 
derivative in $\R^n$ and $\Gamma^h$ is some symmetric tensor field with value 
in $\R^n$ (notice that $\Gamma^h$ is $C^{2k}$ if $h\in \boG^{2k+1}$). In order 
to control the existence time of the geodesic, the geodesic won't be 
parametrized by arc length.

Let $\boS^k=\{(F,Y)\in C^k([0,3],\R^n)\times C^k([0,3],\R^n)\mid \forall 
t\in[0,3], 
F(t)\in U,\ F^x(t)=tF^z(t),\ Y^x(t)>\frac12\text{ and }Y^z<\frac1{16} \}$ and 
$T\boS^k= \{(F,Y)\in 
C^k([0,3],\R^n)\times C^k([0,3],\R^n)\mid \forall t\in[0,3],F^x(t)=tF^z(t) \}$. 
For a metric $h$ and $(q,z)\in B^3\times (\frac14,2A+2)$, we denote by 
$N(h,q,z)$ the unit normal vector to $\{x=0\}$ at $(q,0,z)$. We then define the 
map
\[
C:\begin{array}{ccc}
\boG^{2k+1}\times B_3\times(\frac14,2+2A)\times \boS^k&\to& T\boS^k\\
(h,q,z,(F,Y))&\mapsto & \left(t\mapsto M(t)\right)
\end{array}
\]
where
\[
M(t)=\begin{pmatrix}
F(t)-(q,0,z)-\int_0^t\frac{F^z(s)}{Y^x(s)-sY^z(s)}(Y^\top(s)+Y^z(s)\partial_z)ds
 +t\int_0^t\frac{F^z(s)}{Y^x(s)-sY^z(s)}Y^z(s)\partial_xds\\
Y(t)-N(h,q,z)+\int_0^t\frac{F^z(s)}{Y^x(s)-sY^z(s)} 
\Gamma_{F(s)}^h(Y(s),Y(s))ds
\end{pmatrix}
\]
Notice that the assumptions on $(F,Y)\in\boS^k$ give 
$Y^x-sY^z>\frac14F^z>0$. 
The coefficient $\frac{F^z(s)}{Y^x(s)-sY^z(s)}$ and the expression of $M$ 
ensures that $C(h,q,z,(F,Y))$ 
belongs to $T\boS^k$. Moreover $C(h,q,z,(F,Y))=0$ if and only if
\[
\begin{cases}
\frac{d}{dt}F=\frac{F^z(t)}{Y^x(t)-tY^z(t)}Y(t)\\
F(0)=(q,0,z)\\
\frac{D^h}{dt}Y=0\\
Y(0)=N(h,q,z)
\end{cases}
\]
in other words, $F$ runs through the geodesic starting from $(q,0,z)$ in the 
direction $N(h,q,z)$ and $Y(t)$ is the unit tangent vector to this geodesic at 
$F(t)$ and $F$ stops when it crosses $\{x=3z\}$.

For the metric $h=g_\alpha$, let 
$F_\alpha(q,z,t)=(q,\frac{tz}{1+t\cos\alpha},\frac{z}{1+t\cos\alpha})$ and 
recall the notation 
$X_\alpha=\frac1{\sin\alpha}(\partial_x-\cos\alpha\partial_z)$. For 
$(q,z)\in B_3\times(\frac14,2+2A)$, $F_\alpha(q,z,\cdot)\in \boS^k$ and 
$C(g_\alpha,q,z,(F_\alpha(q,z,\cdot),X_\alpha))=0$ (notice that 
$\Gamma^{g_\alpha}=0$). The map $C$ is $C^k$ and the derivative with 
respect to $(F,Y)$ at $(F_\alpha(q,z,\cdot),X_\alpha)$ is:
\[
D_{F,Y}C(g_\alpha,q,z,F_\alpha(q,z,\cdot),X_\alpha):\begin{array}{ccc}
T\boS^k&\to &T\boS^k\\
(G,Z)&\mapsto&\left(t\mapsto\begin{pmatrix}
G(t)-\int_0^t L(s,G(s),Z(s))ds\\
Z(t)\end{pmatrix}\right)
\end{array}
\]
where $L:[0,3]\times \R^n\times \R^n\to\R^n$ is $C^\infty$ in its first 
variable and bilinear in the last two. As 
above, 
this differential is invertible. So there is a neighborhood $\boN$ of 
$g_{\bar \alpha}$ 
in
$\boG^{2k+1}$ and a $C^k$ map $\zeta:\boN\times B_{5/2}\times 
(\frac14,2+2A)\to \boS$ 
such 
that $C(h,q,z,\zeta(h,q,z))=0$. $\zeta$ can be written as 
$\zeta=(\zeta_F,\zeta_Y)$. We have $\zeta_F( 
g_{\bar\alpha},q,z)=F_{\bar \alpha}(q,z,\cdot)$. 
The map $(q,z,t)\mapsto F_{\bar \alpha}(q,z,t)$ is a diffeomorphism from 
$B_{5/2}\times (\frac14,2+2A)\times[0,3]$ onto a subset that contains $V$ 
(actually there 
are no focal time along any geodesic for $ g_{\bar\alpha}$) (see 
Figure~\ref{fig:fig2}). As a 
consequence, $(q,z,t)\mapsto \zeta_F(h,q,z)(t)$ is a diffeomorphism from 
$B_{5/2}\times (\frac14,2+2A)\times[0,3]$ onto a subset that contains $V$ for 
$h$ close to 
$ g_{\bar \alpha}$. So, reducing $\boN$, we have a $C^k$ map $\Xi:\boN\to 
C^k(B_{5/2}\times(\frac14,2+2A)\times [0,3],\R^n)$ defined by $\Xi(h) : 
(q,z,t)\mapsto 
\zeta_F(h,q,z)(t)$ such that $\Xi(h)$ is a diffeomorphism on its image and this 
image contains $V$. As a consequence, for any $h\in \boN$, the converse map 
$\Xi(h)^{-1}$ is well defined on $V$, so we have a $C^k$ map $\widetilde\Xi: 
\boN\to C^k (V,B_{5/2}\times(\frac14,2+2A)\times [0,3])$ such that 
$\widetilde\Xi(h)=\Xi(h)^{-1}$. Then, for $h\in \boN$, we can define the 
vector field $Z_h$ on $V$ by $Z_h(p)=\zeta_Y(h,\widetilde{\Xi}(h)(p))$ (here we 
write $\zeta_Y(h,q,z)(t)=\zeta_Y(h,q,z,t)$). By 
construction, the map $h\in \boN\mapsto Z_h$ is $C^k$ and $Z_h$ is a unit 
vector field 
tangent to geodesic of $h$ starting orthogonally from $\{x=0\}$. As a 
consequence 
$Z_h$ is the gradient of the distance function $\delta_h$ to $\{x=0\}$ for the 
metric $h$.

As a consequence 
for any $r,s$ small, there is a vector field $X_{r,s}$ which is the gradient of 
the distance function to $\{x=0\}$ with respect to $g_{r,s}$. Moreover 
$(r,s)\mapsto X_{r,s}$ is $C^k$ (for any $k$).

Recall that we are interested in the sign of the mean curvature of $S_{s_n}$ at 
$p_n$. We know that $p_n$ belongs to $E$ so $p_n=\phi_{s_nt_n}(\tilde p_n)$ for 
some $\tilde p=(\tilde q_n,s_n,s_n)$ and $t_n\in[0,1]$. Since $p_n\to \bar p$, 
$\tilde p_n\to 
\bar p$ and we can write $\tilde p_n=\bar p+r_n$. Since $J_{r_n,s_n}$ is a 
local homothety of factor $s_n$ from $(U,g_{r_n,s_n})$ into a neighborhood of 
the origin in $(\R^{n-2}\times\R_+^2,g)$, the mean 
curvature of $S_{s_n}$ is then given by 
$\frac1{s_n}\boH(g_{r_n,s_n},X_{r_n,s_n},t_n)$. So we need to understand 
the sign of $\boH(g_{r,s},X_{r,s},t)$ for $r,s$ small and $t\in[0,1]$. This 
quantity is continuous in $(r,s)$ so let us first consider the case 
$(r,s)=(0,0)$.

\subsection{Computation of $\boH( g_{\bar\alpha},X_{\bar \alpha},t)$}
We consider a slightly more general computation. Let $h$ be a constant metric 
of 
the form $h=ds^2+dx^2+2\cos\bar\alpha dxdz+dz^2$ where $ds^2$ is a scalar 
product on $\R^{n-2}$. Let us notice that $\nabla^h\delta_h=X_{\bar \alpha}$. 
Besides a vector field normal (for $h$) to 
$\Psi_{X_{\bar\alpha}}(B_1\times[0,1])$ is given by 
\[
\psi(\frac 
xz)\frac{-\cos\bar\alpha\partial_x+\partial_z}{\sin\bar\alpha}+(1-\psi(\frac
 xz))\partial_z=-\psi(\frac 
xz)\frac{\cos\bar \alpha}{\sin\bar\alpha}\partial_x+(\frac{\psi(\frac 
xz)}{\sin\bar\alpha}+(1-\psi(\frac 
xz))\partial_z
\]
whose norm in the metric $h$ is
\[
W=(\psi^2(\frac xz)+(1-\psi(\frac 
xz))^2+2\sin\bar\alpha \psi(\frac xz)(1-\psi(\frac xz)))^{1/2}
\]
In the following, we will lighten the notations by writing $u=\frac xz$, 
$\psi=\psi(u)$ and $\psi'=\psi'(u)$. So the unit normal vector to 
$\Psi_{X_{\bar\alpha}}(B_1\times[0,1])$ is given by
\[
N=\frac1{W}\Big((-\psi\frac{\cos\bar 
\alpha}{\sin\bar\alpha}\partial_x+(\frac{\psi}{\sin\bar\alpha}+(1-\psi))\partial_z\Big)
\]
We then can compute the mean curvature as the opposite of the 
divergence of $N$. We have
\begin{gather*}
\partial_x W^{-1}=-\frac1z\psi'(2\psi-1)(1-\sin\bar\alpha)W^{-3}\\
\partial_z
W^{-1}=\frac uz\psi'(2\psi-1)(1-\sin\bar\alpha)W^{-3}
\end{gather*}
Thus
\begin{align*}
\Div 
N=&\frac1z 
\psi'(2\psi-1)(1-\sin\bar\alpha)W^{-3} 
(\psi\frac{\cos\bar\alpha}{\sin\bar\alpha} 
+u(\frac{\psi}{\sin\bar\alpha}+(1-\psi)))\\
 &\quad-\frac1z\psi'W^{-1}(\frac{\cos\bar\alpha}{\sin\bar\alpha} 
 +u(\frac1{\sin\bar\alpha}-1))\\
=&\frac{\psi'}z 
\psi'(2\psi-1)(1-\sin\bar\alpha)W^{-3}(\psi(\frac{\cos\bar\alpha}{\sin\bar\alpha}
 +u(\frac1{\sin\bar\alpha}-1))+u)\\
 &\quad-\frac1z\psi'W^{-1}(\frac{\cos\bar\alpha}{\sin\bar\alpha} 
 +u(\frac1{\sin\bar\alpha}-1))\\
=&\frac{\psi'}zW^{-3}\left(u(2\psi-1)(1-\sin\bar 
\alpha)\vphantom{W^2}\right.\\
&\qquad\qquad\qquad\qquad\left.+ (\frac{\cos\bar\alpha}{\sin\bar\alpha} 
 +u(\frac1{\sin\bar\alpha}-1))(\psi(2\psi-1)(1-\sin\bar\alpha)-W^2)\right)\\
=&\frac{\psi'}zW^{-3}\left(u(2\psi-1)(1-\sin\bar\alpha) 
+\left(\frac{\cos\bar\alpha}{\sin\bar\alpha}+
u(\frac1{\sin\bar\alpha}-1)\right)(\psi(1-\sin\bar\alpha)-1)\right)\\
=&\frac{\psi'}zW^{-3}\mu(\psi)
\end{align*}
where $\mu$ is some affine function depending on $u$. So in order to control 
the 
sign of $\mu$ on $[0,1]$, we just have to look at
\[
\mu(0)=-u(1-\sin\bar\alpha)-\left(\frac{\cos\bar\alpha}{\sin\bar\alpha}+
u(\frac1{\sin\bar\alpha}-1)\right)<0\text{ if }\bar\alpha\in(0,\frac\pi2)
\]
and
\[
\mu(1)=u(1-\sin\bar\alpha)-\sin\bar\alpha\left(\frac{\cos\bar\alpha}{\sin\bar\alpha}+
u(\frac1{\sin\bar\alpha}-1)\right)=-\cos\bar\alpha<0\text{ if 
}\bar\alpha\in(0,\frac\pi2)
\]
Thus we have that the mean curvature vanishes for $\bar\alpha=\frac\pi2$ and 
has the opposite sign to $\psi'(\frac xz)$ when $\bar\alpha\in(0,\frac\pi2)$. 
Because of the assumptions on $\psi$, it gives that
\begin{equation}\label{eq:boH}
\boH(h,X_{\bar \alpha},t)>0 \iff \bar\alpha\in(0,\frac\pi2)\text{ and 
}t\neq 0
\end{equation}
As a consequence, for $h= g_{\bar\alpha}$, if $\bar \alpha\in (0,\frac\pi2)$ 
and 
$t_n\to \bar t\in(0,1]$, then the mean curvature of $S_{s_n}$ at 
$p_n$ is positive for large $n$ (and even goes to $+\infty$).

\subsection{The case $\bar t=0$} Let us 
denote $\alpha_n=\alpha(r_n)$, we have
\[
\boH(g_{r_n,s_n},X_{r_n,s_n},t_n)
=\boH(g_{r_n,s_n},X_{r_n,s_n},0)+\int_0^{t_n}\partial_t 
\boH(g_{r_n,s_n},X_{r_n,s_n},s)ds
\]

Since $\boH$ is $C^2$, there is a constant $C$ such that
\[
|\partial_t\boH(g_{r_n,s_n},X_{r_n,s_n},s)-\partial_t 
\boH(g_{r_n,0},X_{r_n,0},s)|\le 
C(|g_{r_n,s_n}-g_{r_n,0}|+|X_{r_n,s_n}-X_{r_n,s_0}|)
\]
By definition of $g_{r,s}$, $|g_{r_n,s_n}-g_{r_n,0}|\le Cs_n$ and 
$|X_{r_n,s_n}-X_{r_n,0}|\le Cs_n$ (see \ref{sec:mettodist}, regularity of the 
map $h\mapsto \nabla^h\delta_h$). This implies that
\begin{align*}
\boH(g_{r_n,s_n},X_{r_n,s_n},t_n)
&\ge\boH(g_{r_n,s_n},X_{r_n,s_n},0)+\int_0^{t_n}\partial_t 
\boH(g_{r_n,0},X_{r_n,0},s)ds-C\int_0^{t_n}s_nds\\
&\ge s_nH(r_n,s_n,s_n)+\boH(g_{r_n,0},X_{r_n,0},t_n)-Ct_ns_n\\
&\ge s_n(H(r_n,s_n,s_n)-Ct_n)\\
&\ge 0\quad\text{for large }n
\end{align*}
where we use $\boH(g_{r_n,0},X_{r_n,0},t_n)\ge 0$ (see \eqref{eq:boH} and 
$\alpha(r_n)\in(0,\frac\pi2]$), 
$H$ positive and $t_n\to 0$.

\subsection{The case $\bar\alpha=\frac\pi2$}

By \eqref{eq:boH}, $\boH(g_{q_n,0},X_{q_n,0},t_n)\ge 0$ so, if 
$\boH(g_{q_n,s_n},X_{q_n,s_n},t_n)\le 0$, we have
\[
0\ge \lim 
\frac{\boH(g_{q_n,s_n},X_{q_n,s_n},t_n)-\boH(g_{q_n,0},X_{q_n,0},t_n)}{s_n}
\]
Since $\boH$ is $C^1$, this limit is
\[
D_h\boH(g_{0,0},\partial_x,\bar t)(\frac{\partial}{\partial 
s}{g_{0,s}}_{|s=0})+D_X\boH(g_{0,0},\partial_x,\bar 
t)(\frac{\partial}{\partial s}{X_{0,s}}_{|s=0})
\]
So we have to compute the two above terms. We have
\begin{align*}
D_h\boH(g_{0,0},\partial_x,\bar t)(\frac{\partial}{\partial 
s}{g_{0,s}}_{|s=0})&=\frac{\partial}{\partial 
s} \boH(g_{0,s},\partial_x,\bar t)_{|s=0}\\
&=\frac{\partial}{\partial 
s} (sH(0,s(1-\bar t),s))_{|s=0}=H(0)
\end{align*}

For the second term, let us study $D_X\boH(g_{0,0},\partial_x,\bar 
t)$. Let $Y$ be a vector field in $V$ such that $Y^z=0$. Thus, for any $s$,
$\partial_x+sY$ and $\boY(\partial_x+sY)$ has no component in the $\partial_z$ 
direction. Thus $\boH(g_{0,0}, \partial_x+sY, \bar t)$ is just the mean 
curvature of $\{z=1\}$ for the metric $g_{0,0}$ so it vanishes. Thus 
$D_X\boH(g_{0,0},\partial_x,\bar t)(Y)=0$. Hence we may focus on the 
$\partial_z$ component of 
$\barre Y=\frac{\partial}{\partial s}{X_{0,s}}_{|s=0}$.

First let us notice that
\[
\frac{\partial}{\partial 
s}{g_{0,s}}_{|s=0}(q,x,z)=D_qg(0)(q)+D_xg(0)x+D_zg(0)z
\]
Let us denote by $k$ the above symmetric tensor. In order to compute $\barre 
Y$, let us remark that if $Z_h$ is the gradient of the distance associated to 
the metric $h$, we have, on $V_0$,
\[
\begin{cases}
h(Z_h,Z_h)=1\\
h(Z_h,\partial_z)=0\\
h(Z_h,\partial_i)=0&\text{for all } 1\le i\le n-2
\end{cases}
\]
and, on $V$, $\nabla^h_{Z_h}Z_h=0$. So differentiating these equalities with 
respect to $s$, for $h=g_{0,s}$, gives
\begin{gather}
k(\partial_x,\partial_x)+2g_{0,0}(\barre Y,\partial_x)=0\text{ on }V_0\\
k(\partial_x,\partial_z)+g_{0,0}(\barre Y,\partial_z)=0\text{ on 
}V_0\label{eq:interest1}\\
k(\partial_x,\partial_i)+g_{0,0}(\barre Y,\partial_i)=0\text{ on }V_0\text{ and 
}1\le 
i\le n-2\\
\Gamma(k)(\partial_x,\partial_x)+\nabla^{g_{0,0}}_{\partial_x}\barre Y=0\text{ 
in }V\label{eq:interest2}
\end{gather}
where $\Gamma(k)(\partial_x,\partial_x)=\sum_m(\partial_x 
k_{xm}-\frac12\partial_mk_{xx})\partial_m$. Let us remark that 
$\Gamma(k)(\partial_x,\partial_x)$ is a constant vector field in $V$. Its 
component along $\partial_z$ is given by 
\[
(\partial_xg(\partial_x,\partial_z)-\frac12\partial_z 
g(\partial_x,\partial_x))_{|p=0}=g(\nabla_{\partial_x}\partial_x,\partial_z)_{|p=0}
\]
So, in order to evaluate the component $\barre Y^z$, we just have to know it 
along $V_0$. Let us first remark that, since $\alpha(0)=\frac\pi2$ is a maximum 
value of $\alpha$, we have $D_q\alpha(0)=0$ and thus $D_qg(0)$ has a vanishing 
$xz$ component. Thus, on $V_0$, 
$k(\partial_x,\partial_z)=z\partial_zg(\partial_x,\partial_z)_{|p=0}$ and, by 
\eqref{eq:interest1}, $\barre Y^z=-zg(\partial_x,\partial_z)_{|p=0}$ along 
$V_0$. Let us denote by $b=\partial_zg(\partial_x,\partial_z)_{|p=0}$ and 
$a=g(\nabla_{\partial_x}\partial_x,\partial_z)_{|p=0}$. Thus 
\eqref{eq:interest2} gives $\barre Y^z=-ax-bz$. So the question is then to 
compute
\[
D_X\boH(g_{0,0},\partial_x,\bar t)((-ax-bz)\partial_z)
\]
we have 
\[
\boY(\partial_x-s(ax+bz)\partial_z)=-\partial_x+s(1-\psi)(ax+bz)\partial_z
\]
A vectorfield normal to the associated hypersurface is given by
\[
s(1-\psi)(ax+bz)\partial_x+\partial_z
\]
whose norm is given by
\[
W=(1+s^2(1-\psi)^2(ax+bz)^2)^{1/2}
\]
So the unit normal is given by 
\[
N=\frac1W(s(1-\psi)(ax+bz)\partial_x+\partial_z)=\partial_z 
+s(1-\psi)(ax+bz)\partial_x+O(s^2)
\]
Thus
\begin{align*}
\Div N=&s\left(-\frac1z\psi'(ax+bz)+(1-\psi)a\right)+O(s^2)\\
=&-s(\psi'(au+b)-(1-\psi) a)+O(s^2)
\end{align*}
So the derivative with respect to $s$ is given by $-(\psi'(au+b)-(1-\psi) a)$. 
Thus 
the sign of $D_X\boH(g_{0,0},\partial_x,\bar t)((-ax-bz)\partial_z)$ is the one 
of $\psi'(au+b)-(1-\psi) a$. This expression is non-negative if
\[
b\ge (\frac{1-\psi}{\psi'}-u)a
\]
Notice that $\psi$ can be chosen such that $\frac{1-\psi}{\psi'}-u$ is bounded 
on $[0,1]$ (take $\psi(u)=1-\exp(-1/(u-1)^2)$ near $1$). Now by the coordinate 
improvement \eqref{eq:coord_improve}, the above inequality can be assumed to be 
true. 

Finally this implies that 
\[
D_h\boH(g_{0,0},\partial_x,\bar t)(\frac{\partial}{\partial 
s}{g_{0,s}}_{|s=0})+D_X\boH(g_{0,0},\partial_x,\bar 
t)(\frac{\partial}{\partial s}{X_{0,s}}_{|s=0})\ge H(0)>0
\]
contradicting $\boH(g_{q_n,s_n},X_{q_n,s_n},t_n)\le 0$. We then have proved 
that the mean curvature of $S_s$ is positive for small $s$ and have finished 
the 
proof of Proposition~\ref{prop:app}.

\begin{figure}
\centering
\resizebox{0.6\linewidth}{!}{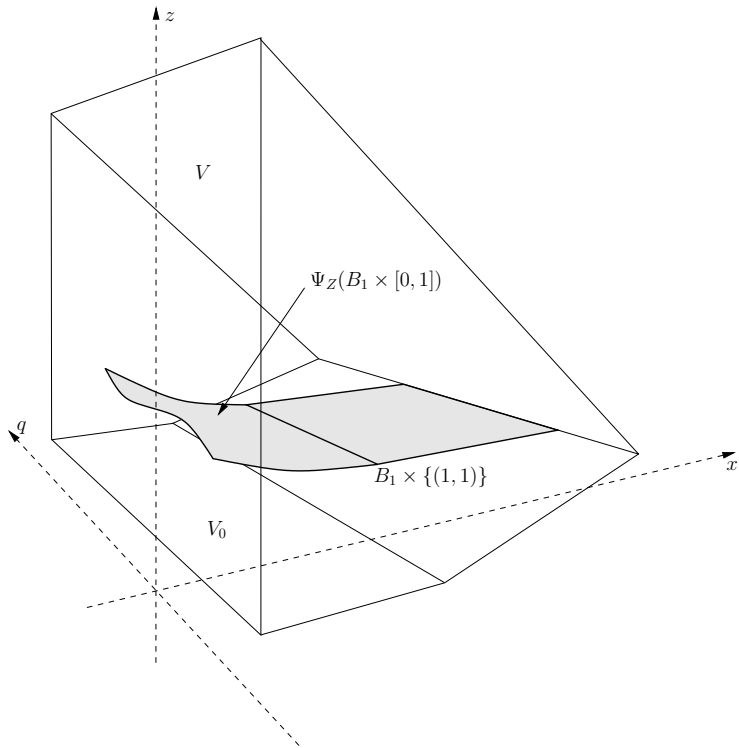}
\caption{The hypersurface $S_s$ in $V$}\label{fig:fig1}
\end{figure}

\bibliographystyle{amsplain}
\bibliography{bibliography.bib}

\end{document}